\documentclass [a4paper,12pt]{article}
\usepackage{amsmath}%用于Align环境
\usepackage{amsfonts}%用于\mathbb{}
\usepackage{indentfirst}
\usepackage{esint}
\usepackage{latexsym}
\usepackage{bbm}
\usepackage{amsfonts}
\usepackage{mathrsfs}
\usepackage{amssymb}
\usepackage{amsmath}
\usepackage{amsthm}
\usepackage{latexsym}
\usepackage{indentfirst}
\usepackage{bm}
\usepackage{esint}
\usepackage{hyperref} 
\hypersetup{hidelinks}
\usepackage{cleveref}
\usepackage{cite}
\usepackage[numbers,square]{natbib}
\numberwithin{equation}{section}
\usepackage{color}
\allowdisplaybreaks%用于自动断页
\usepackage{amsthm}%用于定理环境
\usepackage{amssymb}%第一次用于写右箭头\nrightarrow命令
\usepackage{enumerate}%计时器，用于编号公式之类;还用于列表环境。
\usepackage{ulem}%用于划删除线
\allowdisplaybreaks

\newtheorem{theorem}{Theorem}[section]
\newtheorem{lemma}[theorem]{Lemma}
\newtheorem{thm}[theorem]{Theorem}

\newtheorem{defn}[theorem]{Definition}
\newtheorem{rmk}[theorem]{Remark}

\makeatletter%以下四行是用于写大小写罗马数字的，本文暂时只用得到大写。
\usepackage{amsmath}%用于Align环境
\usepackage{amsfonts}%用于\mathbb{}
\usepackage{indentfirst}
\usepackage{esint}
\usepackage{latexsym}
\usepackage{authblk}
\usepackage{color}
\UseRawInputEncoding
\usepackage{amsthm}%用于定理环境
\usepackage{amssymb}%第一次用于写右箭头\nrightarrow命令
\usepackage{enumerate}%计时器，用于编号公式之类;还用于列表环境。
\usepackage{ulem}%用于划删除线
\makeatletter%以下四行是用于写大小写罗马数字的，本文暂时只用得到大写。

\newcommand{\Rmnum}[1]{\expandafter\@slowromancap\romannumeral #1@}

\makeatother
\begin{document}
\title{Large-scale boundary estimates of parabolic homogenization over rough boundaries}
\author[a]{Pengxiu Yu\thanks{Email: Pxyu@ruc.edu.cn}}
\author[b,c]{Yiping Zhang\thanks{Corresponding author, Email: zhangyiping161@mails.ucas.ac.cn}}
\affil[a]{\footnotesize{School of Mathematics,
Renmin University of China, Beijing 100872, China}}
\affil[b]{\footnotesize{School of Mathematics and Statistics, and Hubei Key Lab--Math. Sci., Central China Normal University, Wuhan 430079, China}}
\affil[c]{\footnotesize{Key Laboratory of Nonlinear Analysis \& Applications (Ministry of Education), Central China Normal University, Wuhan 430079, China}}

%\author{Pengxiu Yu\thanks{School of Mathematics and Statistics, Central China Normal University,
%Wuhan 430079, Hubei, People's Republic of China. (zhangyiping161@mails.ucas.ac.cn).},\quad   Yiping Zhang\thanks{School of Mathematics and Statistics, Central China Normal University,
%Wuhan 430079, Hubei, People's Republic of China. (zhangyiping161@mails.ucas.ac.cn).}}

\date{}
\maketitle
\begin{abstract}
In this paper, for a family of second-order parabolic system or equation with rapidly oscillating and time-dependent periodic coefficients over rough boundaries, we obtain the large-scale boundary estimates, by a quantitative approach.
The quantitative approach relies on approximating twice: we first approximate the original parabolic problem over rough boundary by the same equation over a non-oscillating boundary and then approximate the oscillating equation over a non-oscillating boundary by its homogenized equation over the same non-oscillating boundary.
\end{abstract}

\section{Introduction}\label{i1}
\let\thefootnote\relax\footnotetext{ Mathematics Subject Classification: 35B27, 35K15.\\
Keywords: Homogenization, Parabolic equations, rough domain, large-scale boundary regularity.}
\subsection{Motivation}\label{m1.1}
In this paper, we want to investigate the boundary regularity for parabolic system/equation of divergence form in a bounded domain in $\mathbb{R}^{d+1}$ whose space-boundary is arbitrarily at small scales. More precisely, let $\Omega^\varepsilon$ be a bounded domain in $\mathbb{R}^d$ and $0\in \partial\Omega^\varepsilon$.
Denote $Q_r=:B_r\times (-r^2,0)$, $\Omega^\varepsilon_r=:\Omega^\varepsilon \cap B_r$,
$Q_r^\varepsilon=:\Omega^\varepsilon_r\times (-r^2,0)$ and
$\Delta_r^\varepsilon=:(\partial\Omega^\varepsilon \cap B_r)\times [-r^2,0]$. Now, for
$u_\varepsilon\in L^2(-4,0; H^1(\Omega^\varepsilon_{2}))$ with $\partial_t u_\varepsilon\in
L^2(-4,0;H^{-1}(\Omega^\varepsilon_{2}))$, we consider the following linear parabolic system/equation:
\begin{equation} \label{1.1}
\left\{\begin{aligned}
\partial_t u_\varepsilon - \operatorname{div}\left(A^\varepsilon \nabla u_\varepsilon \right) &= 0 \quad \text{ in } Q_{2}^\varepsilon,\\
u_\varepsilon&=0 \quad \text{ on } \Delta_{2}^\varepsilon,\\
\end{aligned}\right.
\end{equation}
where $A^\varepsilon(x,t)=:A(x/\varepsilon,t/\varepsilon^2)$ and the definition of $\Omega^\varepsilon$ with rapidly oscillating boundaries will be given in Definition \ref{d1.1}.

For the uniform regularity estimates in parabolic periodic homogenization, Shen and Gen \cite{MR3361284,MR4072212} obtained the
large-scale interior/boundary Lipschitz estimate in a series of papers. In particular, for the case of non-self-similar scales 
\cite{MR4072212} (i.e., $A^\varepsilon(x,t)=A(x/\varepsilon,t/\varepsilon^k), k\in (0,+\infty)$), by a quantitative approach, they obtained that if $A$ is periodic and uniform elliptic, and the non-oscillating domain $\Omega_2\times(-4,0)$ is $C^{1,\alpha}$, then for any $\varepsilon+\varepsilon^{k/2}\leq r<2$,
\begin{equation}\label{1.2}
\left(\fint_{\Omega_r\times(-r^2,0)}|\nabla u_\varepsilon|^2\right)^{1/2}\leq C \left(\fint_{\Omega_2\times(-4,0)}|\nabla 
u_\varepsilon|^2\right)^{1/2},
\end{equation}
for the constant $C$ independent of $r$ and $\varepsilon$. For $k=2$, by a simple blow-up argument, the large-scale boundary 
Lipschitz estimates \eqref{1.2} can by improved as the point-wise boundary Lipschitz estimates if $A$ is H\"older continuous. 
Moreover, we refer readers to \cite{MR3812862} for the quantitative estimates in stochastic homogenization of parabolic version.

The literature on uniform regularity estimates of the elliptic version is extensive. The seminal work was conducted in the late 1980s by Avellaneda and Lin, who published a series of papers on both divergence and non-divergence forms in periodic homogenization \cite{MR910954,MR978702}. They employed a compactness method for their research. For the almost-periodic or stochastic homogenization in the elliptic case of the divergence form, uniform regularity estimates can be found in various sources, including \cite{MR3519974,MR3437852,MR3541853,MR3399132,MR3744924}.
Specifically, in the context of the Neumann problem, Kenig, Lin, and Shen \cite{MR3073881} initially derived Lipschitz estimates for symmetric coefficients using the compactness method. Subsequently, Armstrong and Shen \cite{MR3541853} removed the symmetric assumption through a quantitative approach.
Recent investigations have focused on uniform regularity estimates in various equations within the realm of homogenization theory. These include the Stokes equations over John domains \cite{MR4619004}, multi-scale modelings \cite{MR4566686,MR4708661}, Darcy's laws in periodically perforated domains \cite{MR4438913}, and the degenerate elliptic case in perforated domains \cite{shen2023uniform} and so on.

%There is a vast literature in uniform regularity estimates of elliptic version. The pioneering work
%dates back to the late 1980s by Avellaneda and Lin in series of papers for both divergence and
%non-divergence forms in periodic homogenization \cite{MR910954,MR978702}, using a compactness method. For the almost-periodic or stochastic homogenization in elliptic case of divergence form, one can find the uniform regularity estimates in, e.g., \cite{MR3519974,MR3437852,MR3541853,MR3399132,MR3744924}. Specially, for the Neumann problem, using the compactness mathod,
%Kenig, Lin and Shen \cite{MR3073881} first obtained the Lipschitz estimates for symmetric coefficients. And the symmetric assumption was finally removed by Armstrong and Shen in \cite{MR3541853} by a quantitative method. Recently,
%the uniform regularity estimates have been investigated in various kinds of equations in homogenization theory, such as the Stokes equations over John domains \cite{MR4619004}, multi-scale modelings \cite{MR4566686,MR4708661}, the Darcy's laws in periodically perforated domains \cite{MR4438913}, the degenerate elliptic case in perforated domains \cite{shen2023uniform} and so on.

In physical reality, it is natural to consider the PDEs (especially for the fluid dynamics) with rapidly oscillating boundaries. One of the main goals is to determine the effective boundary conditions (also called the wall laws) and obtain an higher-order convergence rate. We refer readers to \cite{MR1657773,MR1906814,MR2410410,MR3642228,MR2470924,MR4715064} for this topic and the 
reference therein for more details. Moreover, another main goal is to obtain  the large-scale regularity, which should be expected and observed in physical reality for the solutions to the PDEs with rapidly oscillating boundaries.
For the elliptic case of divergence form with rapidly oscillating boundaries given by $x_d=\varepsilon \psi(x'/\varepsilon)$, using the compactness method, Kenig and Prange obtained the large-scale
Lipchitz estimate under the assumption $\psi\in W^{1,\infty}(\mathbb{R}^{d-1})$ \cite{MR3784804} and $\psi\in C^{1,\alpha}(\mathbb{R}^{d-1})$ \cite{MR3325774}, respectively. Recently, by a quantitative method, Zhuge \cite{MR4199276} extended the above result to a more general domain, which satisfies the so-called $\varepsilon$-flatness with a
$\sigma$-admissible modulus. Moreover, for the large-scale regularity estimates of the Stokes systems and  stationary Navier-Stokes equation over bumpy boundaries, we refer to \cite{MR4419614,MR4123336,MR4702317} for the details.

In this paper, we want to extend the result obtained by Zhuge \cite{MR4199276} to the parabolic case. Firstly, we introduce the concept of $\varepsilon$-flatness with a $\sigma$-admissible modulus, cited from Zhuge \cite[Definitions 1.1-1.2]{MR4199276}.

\begin{defn}\label{d1.1}
Let $\Omega^\varepsilon$ be a bounded domain in $\mathbb{R}^d$ with $\varepsilon>0$. We say $\Omega^\varepsilon$ is $\varepsilon$-scale flat with a modulus $\zeta:(0,1]\times (0,1]\mapsto [0,1]$, if for any $y\in \partial \Omega^\varepsilon$ and $r\in (\varepsilon,1)$, there exists a unit (outward normal) vector $n_r = n_r(y) \in \mathbb{R}^d$ so that
	\begin{equation}\label{1.3}
	\begin{aligned}
	& B_r(y) \cap \{ x\in \mathbb{R}^d: (x-y)\cdot n_r < -r\zeta(r,\varepsilon/r) \} \\[5pt]
	& \qquad \subset \Omega^\varepsilon_r(y) \subset B_r(y) \cap \{ x\in \mathbb{R}^d: (x-y)\cdot n_r < r\zeta(r,\varepsilon/r) \}.
	\end{aligned}
	\end{equation}
\end{defn}

Moreover, the modulus above should satisfy an additional quantitative condition, which is stated as follows:

\begin{defn}\label{d1.2}
Let $\eta:(0,1]\times (0,1] \to (0,1]$ be a continuous function. We say that $\eta$ is an ``admissible modulus'' if the following conditions hold:
	\begin{itemize}
		\item Flatness condition:
		\begin{equation}\label{1.4}
		\lim_{t\to 0^+} \sup_{r,s\in (0,t)} \eta(r,s) = 0.
		\end{equation}

		\item A Dini-type condition:
		\begin{equation*}
		\lim_{t\to 0^+} \sup_{\varepsilon \in (0,t^2)} \int_{\varepsilon/t}^{t} \frac{\eta(r,\varepsilon/r)}{r} dr = 0.
		\end{equation*}
	\end{itemize}
	Moreover, we say $\eta$ is ``$\sigma$-admissible'' if $\eta^\sigma$ is an admissible modulus.
\end{defn}

As pointed out in \cite{MR4199276}, there are three typical cases of $\sigma$-admissible moduli in the following.\\

\noindent Case 1: $\zeta (r,s)=Cr^\alpha$ if $\Omega^\varepsilon$ is uniformly $C^{1,\alpha}$.\\

\noindent Case 2: $\zeta (r,s)=Cs$ if the boundary $\partial\Omega^\varepsilon$ is locally given by the graph of $x_d=\varepsilon \psi(x'/\varepsilon)$ with $\psi\in C^0 \cap L^\infty(\mathbb{R}^{d-1})$. Moreover, $\zeta (r,s)=Cs^{1-\alpha}$ if $\psi(x')$ is uniformly $C^\alpha$-H\"older continuous in $\mathbb{R}^{d-1}$ (not necessarily bounded).\\

\noindent Case 3: $\zeta (r,s)=Cr^\beta+s^\alpha$ if the oscillation boundary is given by a graph $x_d=\psi_0(x')+\varepsilon\psi_1(x'/\varepsilon)$, where $\psi_0$ is a $C^{1,\beta}$ function and $\psi_1$ satisfies either condition in Case 2.\\

Now, we assume that $\Omega^\varepsilon$ is $\varepsilon$-scale flat with a $\sigma$-admissible  modulus
$\varsigma$ for some $\sigma\in (0,1/2)$ and  we turn back to our parabolic problem \eqref{1.1}. Recall that for the elliptic
case over rough boundaries, one can directly extend $u_\varepsilon$ across the rough boundary by
zero-extension which would preserve the $H^1$-norm of $u_\varepsilon$. However, for the parabolic
case, something turns out to be different. The main difference is that we need to find a meaningful
extension of $\partial_t u_\varepsilon$ in the Sobolev space $L^2(-4,0;H^{-1}(\Omega^\varepsilon_2))$ with negative index, which seems  unavoidable such as in the basic energy
estimates and in the parabolic Caccioppoli inequalities \cite[Lemma B.6]{MR3812862}.

Precisely, note that $u_\varepsilon=0$ on $\Delta_2^\varepsilon$, then we can extend $u_\varepsilon$ naturally to $Q_2$ by

\begin{equation}\label{1.5}
\tilde{u}_\varepsilon(x,t)=\left\{\begin{aligned}
&u_\varepsilon(x,t),\quad  &\text{ if }(x,t)\in (\Omega^\varepsilon \cap B_2)\times[-4,0],\\
&0, &\text{ if }(x,t)\in (B_2\setminus\Omega^\varepsilon)\times[-4,0].
\end{aligned}\right.\end{equation}

 It is easy to verify that
$\tilde{u}_\varepsilon\in L^2(-4,0;H^1(B_2))$ with $||\nabla
\tilde{u}_\varepsilon||_{L^2(Q_2)}=||\nabla {u}_\varepsilon||_{L^2(Q_2^\varepsilon)}$ and
$||\tilde{u}_\varepsilon||_{L^2(Q_2)}=|| {u}_\varepsilon||_{L^2(Q_2^\varepsilon)}$. Moreover, we
know that $\nabla \tilde{u}_\varepsilon=0$ and $\partial_t \tilde{u}_\varepsilon=0$ for  $(x,t)\in
(B_2\setminus\Omega_2^\varepsilon)\times[-4,0]$. For simplicity, we will still denote the extended function
$\tilde{u}_\varepsilon$ by $u_\varepsilon$ if the content is understood.

Unfortunately, for the zero-extension in Sobolev space with negative exponent, the following control
$$||\partial_t \tilde{u}_\varepsilon||_{L^2(-4,0;H^{-1}(B_2))}\leq ||\partial_t \tilde{u}_\varepsilon||_{L^2(-4,0;H^{-1}(\Omega^\varepsilon_2))}$$
may not hold true for general function $u_\varepsilon$ and for general domain $\Omega^\varepsilon$.

To proceed, in order to cover a more general domain, we now introduce an abstract framework of hypotheses on this extension, which can be rigorously verified at least for the rough boundary given by a Lipschitz graph, see Lemma \ref{l2.2} for the details.

To be more precise, for the solution $u_\varepsilon$ to the parabolic probelm \eqref{1.1}, assume the following hypothesis hold true:\\

\noindent (H) For any Lipschitz bounded domain $\Omega\subset\mathbb{R}^d$ with $\Omega\cap \Omega^\varepsilon\neq \emptyset$ and any $4>s>0$, there holds
\begin{equation}\label{1.6}
||\partial_t \tilde{u}_\varepsilon||_{L^2(-s,0;H^{-1}(\Omega))}\leq C ||\nabla u_\varepsilon||_{L^2((\Omega^\varepsilon \cap \Omega)\times(-s,0))};
\end{equation}

%\noindent (H2)\label{h2} For any $4>s>0$, there holds
%\begin{equation}\label{1.7}
%||\partial_t \tilde{u}_\varepsilon||_{L^2(-s,0;H^{-1}(\Omega^\varepsilon))}\leq C ||\nabla u_\varepsilon||_{L^2(\Omega^\varepsilon \times(-s,0))};
%\end{equation}

\noindent where the constant $C$ in \eqref{1.6} depends only on $\Lambda,d$ and $m$. For the meanings of these constants, see Section \ref{m1.2} for the details.

\begin{rmk}
We have noticed the method used in \cite{MR4514952} to avoid the extension in Sobolev space with negative exponent, which may be helpful to remove the assumption (H) (however, the difficulty is that we consider the parabolic problem with the coefficient matrix depending on the time variable). Precisely, they \cite{MR4514952} consider the equation satisfied by $\tilde{U}_\varepsilon(x,t)=:\int_0^t \tilde{u}_\varepsilon(x,s)ds$ which has better regularity in time variable.
\end{rmk}

At the end of this subsection, we give an explanation of using Sobolev-Poinc\'{a}re's inequality over rough boundaries. For any $\varphi\in H^1(\Omega^\varepsilon_r)$ such that $\varphi=0$ on
$\partial\Omega^\varepsilon \cap B_r$ for $r\geq \varepsilon$, we extend it to $B_r$ by zero across the rough boundary
$\partial\Omega^\varepsilon \cap B_r$, then we can use the classical Sobolev-Poinc\'{a}re's inequality for $\varphi$ in the non-oscillating domains $B_r$.
\subsection{Main result}\label{m1.2}

We consider a family of oscillating parabolic operators with self-similar scales in divergence form
\begin{equation}\label{1.8}
\operatorname{div}(A(x/\varepsilon,t/\varepsilon^2)\nabla)=:\frac{\partial}{\partial x_i}\left\{a^{\alpha\beta}_{ij}\left(\frac x\varepsilon,\frac t{\varepsilon^2}\right)\frac{\partial}{\partial x_j}\right\},
\end{equation}
with $1\leq i,j\leq d$, $1\leq \alpha,\beta \leq m$ (the summation convention is used throughout the paper), where $d$ is the space dimension and $m$ is the number of equations. For our purpose, assume that the coefficient matrix $A=(a^{\alpha\beta}_{ij})$ satisfies the following conditions:\\

(i) Ellipticity Condition: For some $\Lambda>1$ and all $(y,s)\in \mathbb{R}^d\times \mathbb{R}$, $\xi=(\xi_i^\alpha)\in \mathbb{R}^{m\times d}$, there holds that
\begin{equation}\label{1.9}
\Lambda^{-1} |\xi|^2\leq a^{\alpha\beta}_{ij}(y,s)\xi_i^\alpha\cdot\xi_j^\beta\leq \Lambda|\xi|^2.
\end{equation}

(ii) Periodicity Condition:
\begin{equation}\label{1.10}
A(y+z,s+t)=A(y,s) \quad \text{for } (y,s)\in \mathbb{R}^d\times \mathbb{R} \text{ and }(z,t)\in\mathbb{Z}^d\times\mathbb{Z}.\end{equation}

Suppose $\left\{\Omega^\varepsilon:\varepsilon>0\right\}$ is a family of bounded domains with oscillating boundaries and $0\in \partial\Omega^\varepsilon$. Recall that we denote $I_r=:(-r^2,0)$, $Q_r=:B_r\times I_r$, $\Omega^\varepsilon_r=:\Omega^\varepsilon \cap B_r$,
$Q_r^\varepsilon=:\Omega^\varepsilon_r\times I_r$,
$\Delta_r^\varepsilon=:(\partial\Omega^\varepsilon \cap B_r)\times [-r^2,0]$ and $A^\varepsilon(x,t)=:A(x/\varepsilon,t/\varepsilon^2)$. Now, we define the weak solution of \eqref{1.1}.
$u_\varepsilon\in L^2(-4,0;H^1(\Omega^\varepsilon_2))$ with $\partial_t u_\varepsilon\in L^2(-4,0;H^{-1}(\Omega^\varepsilon_2))$  is a weak solution of \eqref{1.1} if for any $-4\leq s<0$ and for any $\varphi\in C^\infty(s,0;C^\infty_0(\Omega^\varepsilon_2))$, there holds

$$\int_{\Omega^\varepsilon_2\times\{0\}}u_\varepsilon\varphi
-\int_{\Omega^\varepsilon_2\times\{s\}}u_\varepsilon\varphi
-\int^0_s\int_{{\Omega^\varepsilon_2}}u_\varepsilon\partial_t \varphi
+\int^0_s\int_{{\Omega^\varepsilon_2}}A^\varepsilon\nabla u_\varepsilon\nabla \varphi=0,$$

\noindent
where  $u_\varepsilon\varphi\in L^2(s,0;H^1_0(\Omega^\varepsilon_2))$ for any $\varphi\in C^\infty(s,0;C^\infty_0(\Omega^\varepsilon_2))$.\\

Now the main result of this paper is stated as follows:

\begin{thm}\label{t1.4}
Let $\varepsilon\in (0,1)$ and assume that $\Omega^\varepsilon$ is $\varepsilon$-scale flat with a $\sigma$-admissible  modulus $\varsigma$ for some $\sigma\in (0,1/2)$. Let $u_\varepsilon$ be a weak solution of \eqref{1.1}, satisfying the hypothesis \eqref{1.6}. Then for any $r\in (\varepsilon,1)$, there holds the following large-scale boundary Lipschitz estimates
\begin{equation}\label{1.11}
\left(\fint_{Q^\varepsilon_r}|\nabla u_\varepsilon|^2\right)^{1/2}\leq C \left(\fint_{Q^\varepsilon_2}|\nabla u_\varepsilon|^2\right)^{1/2},
\end{equation}
where the constant $C$ depends only on $\Lambda, d, m$ and $\varsigma$.
\end{thm}

By Definitions \ref{d1.1} and \ref{d1.2}, without loss of generality, we may assume $r\varsigma(r,\varepsilon/r)$ is nondecreasing and $\varsigma(r,\varepsilon/r)\leq \frac 1 4$ for $r\in(\varepsilon,1)$, then $\varepsilon\geq 4\varepsilon^*=\varepsilon\varsigma(\varepsilon,1)$.

Ignoring the hypothesis \eqref{1.6} involved the extension in Sobolev spaces with negative exponent, the above result is a generalization of parabolic version of the previous work by Zhuge \cite{MR4199276} in elliptic case, which sates that the large-scale (macroscopic) smoothness of the boundary implies the large-scale (macroscopic) smoothness of the solutions of PDEs with rough boundaries.\\

\subsection{Outline of the proof}
In this subsection, we briefly introduce the ideas and the structure of this paper.

In Section \ref{p2}, we introduce some preliminaries including the averaging operator which is used to improve the large-scale regularity of $\nabla u_\varepsilon$ and the rigorous proof of the hypothesis \eqref{1.6} with the boundary $\partial \Omega^\varepsilon$ given by a Lipschitz graph.

In Section \ref{l3}, following the ideas in \cite{MR3812862}, we introduce the parabolic Caccioppoli inequality over rough boundaries, which yields a large-scale Meyers estimate. Moreover, using the $\varepsilon$-flatness of $\Omega^\varepsilon$ with a $\sigma$-admissible  modulus $\varsigma$ and the large-scale Meyers estimate, we approximate the original parabolic problem \eqref{1.1} over rough boundary by the same equation over a non-oscillating boundary.

In Section \ref{l4}, we approximate the oscillating equation over a non-oscillating boundary investigated in Section 3 by its homogenized equation over the same non-oscillating boundary. Moreover, we can obtain the excess estimates in the above two steps. Finally, we complete the proof of the large-scale boundary Lipschitz estimates \eqref{1.11} by an iteration result and the parabolic Caccioppoli inequality.\\

We end this section with the following remark.

\begin{rmk}\label{r1.5}In this remark, we try to move forward with the coefficients $A$ with non-self-similar scales. In view of the non-self-similar scales considered in parabolic homogenization \cite{MR4072212}, the method used in this paper may be applied to the case of non-self-similar scales and a similar large-scale boundary Lipschitz regularity continues to hold. Actually, in order to obtain the excess estimates, we need to approximate the original oscillating problem over rough boundary twice. We first approximate the original oscillating problem over rough boundary by the same oscillating problem over flat boundary as in Section \ref{l3}, where we only use the ellipticity condition \eqref{1.9} to estimate the errors. Next, we approximate the original oscillating problem over flat boundary by the homogenized problem over flat boundary as in Section \ref{l4}, where the convergence rates obtained in \cite{MR4072212} and the regularity of the solution to the homogenized problem could be used.
\end{rmk}

\section{Preliminaries}\label{p2}
Due to the boundary geometry of $\partial \Omega^\varepsilon$, one can not expect a uniform point-wise boundary Lipschitz estimate. However, the macroscopic (large-scale) smoothness should be expected in real world, which may be a starting point in \cite{MR4199276} to introduce the following averaging operator.

For $p\in [1,\infty)$, define the averaging operator of parabolic version
	\begin{equation*}
		\mathcal{M}_r^p[g](x,t) = \left( \fint_{Q_r(x,t)} |g|^{p} \right)^{1/p},
	\end{equation*}
with $Q_r(x,t)=:B_r(x)\times (t-r^2,t)$.
%Note that if $g=g(x,t)$, then
%\begin{equation*}
%		\mathcal{M}_r^p[g](x,t) = \left( \fint_{B_r(x)} |g(y,t)|^{p}dy \right)^{1/p}.
%	\end{equation*}
Moreover, for convenience, sometimes we write $\mathcal{M}_r^2$ as $\mathcal{M}_r$ throughout the paper. The following lemma collects useful properties of $\mathcal{M}_r$, whose proof can be found in \cite[Lemma 2.1]{MR4702317} and \cite[Lemma 2.3]{MR4199276} for the elliptic version.

\begin{lemma}\label{l2.1}
		For $p\in[1,\infty)$ and $g\in L^p(Q_1)$, there hold the following inequalities:
		\begin{enumerate}
			\item For $1\le p'\le p<\infty$ and $Q_{r}(x,t)\subset Q_1$,
			\begin{equation}\label{2.1}
				\mathcal{M}_r^{p'}[g](x,t) \le \mathcal{M}_r^{p}[g](x,t).
			\end{equation}

			\item For $0< r_1\le r_2<1$ and $Q_{r_2}(x,t)\subset Q_1$,
			\begin{equation}\label{2.2}
				\mathcal{M}_{r_1}^{p}[g](x,t) \le C\left(\frac{r_2}{r_1}\right)^{(d+2)/p} \mathcal{M}_{r_2}^{p}[g](x,t).
			\end{equation}
			\item For $0<r_1\le r_2$ with $Q_{r_1+r_2}(x,t)\subset Q_1$,
			\begin{equation}\label{2.3}
				\int_{Q_{r_2}(x,t)} |g|^p
				\le C \int_{Q_{r_2}(x,t)} \mathcal{M}_{r_1}^p[g]^p
				\le C\int_{Q_{r_1+r_2}(x,t)} |g|^p.
			\end{equation}

			\item For $0<r_1\le r_2\le r$ with $Q_{r+r_1+r_2}(x,t)\subset Q_1$ and $q\in[p,\infty)$,
			\begin{equation}\label{2.4}
				\begin{aligned}
					\fint_{Q_{r}(x,t)} \mathcal{M}^{p}_{r_2} [g]^{q}
					\le C \fint_{Q_{r+r_2}(x,t)} \mathcal{M}^{p}_{r_1} [g]^{q}.
				\end{aligned}
			\end{equation}
			\item  For $0<r_1\le r_2$ with $Q_{r_1+r_2}(x,t)\subset B_1$,
			\begin{equation}\label{2.5}
				\mathcal{M}_{r_2}^p[g](x,t)
				\le C\fint_{Q_{r_1}(x,t)} \mathcal{M}_{r_2}^p[g].
			\end{equation}
	\end{enumerate}
		Here the constant $C$ depends on $p$ and $ p'$, but not on $r, r_1$ or $r_2$.
	\end{lemma}
	
Now, we give a rigorous proof of the hypothesis \eqref{1.6} if $\Omega^\varepsilon$ is a bounded Lipschitz domain (whose Lipschitz character may depend on $\varepsilon$) in $\mathbb{R}^d$.
\begin{lemma}\label{l2.2}
Assume that $\Omega^\varepsilon$ is a bounded Lipschitz domain in $\mathbb{R}^d$ (whose Lipschitz character may depend on $\varepsilon$, for fixed $\varepsilon>0$) and $u_\varepsilon\in L^2(-4,0;\Omega^\varepsilon_{2})$ is a weak solution to the parabolic equation/system \eqref{1.1}, then for any bounded Lipschitz domain $\Omega\subset B_{2}$ and for any $-2<-s<0$, there holds:
\begin{equation}\label{2.6}
||\partial_t \tilde{u}_\varepsilon||_{L^2(-s,0;H^{-1}(\Omega))}\leq C ||\nabla u_\varepsilon||_{L^2(\Omega\times(-s,0))},
\end{equation}
for the constant $C$ depending only on $\Lambda, d$ and $m$.
\end{lemma}
\begin{proof}We need to consider the following three cases: Case (1), $\Omega\subset (B_2\setminus \Omega^\varepsilon)$; Case (2), $\Omega\subset (\Omega^\varepsilon\cap B_2)$; Case (3), $\Omega\cap ( B_2\cap \partial\Omega^\varepsilon)\neq \emptyset$.

It is easy to check that if $\Omega\subset B_2\setminus \Omega^\varepsilon$, then $||\partial_t
\tilde{u}_\varepsilon||_{L^2(-s,0;H^{-1}(\Omega))}=0$. Moreover, if $\Omega\subset
(\Omega^\varepsilon\cap B_2)$, then by equation \eqref{1.1} and the extension \eqref{1.5}, we have
$$||\partial_t \tilde{u}_\varepsilon||_{L^2(-s,0;H^{-1}(\Omega))}=||\partial_t {u}_\varepsilon||_{L^2(-s,0;H^{-1}(\Omega))}\leq C||\nabla u_\varepsilon||_{L^2(\Omega\times(-s,0))}.$$

Therefore, we need only to consider the case $\Omega\cap ( B_2\cap \partial\Omega^\varepsilon)\neq \emptyset$.
For any $\phi\in C^\infty(-s,0;C^\infty_0(\Omega))$, a direct computation shows that
\begin{equation}\label{2.7}\begin{aligned}
\int_{-s}^0\int_{\Omega}\partial_t \tilde{u}_\varepsilon\cdot
\phi=&\int_{\Omega\times\{0\}}\tilde{u}_\varepsilon
\phi-\int_{\Omega\times\{-s\}}\tilde{u}_\varepsilon \phi-\int^0_{-s}\int_{\Omega}
\tilde{u}_\varepsilon\cdot \partial_t\phi\\
=&\int_{(\Omega^\varepsilon_2\cap \Omega)\times\{0\}}{u}_\varepsilon \phi-\int_{(\Omega^\varepsilon\cap
\Omega)\times\{-s\}}{u}_\varepsilon \phi-\int_{-s}^0\int_{\Omega^\varepsilon_2\cap \Omega}
{u}_\varepsilon\cdot \partial_t\phi\\
=&\int_{-s}^0\int_{\Omega^\varepsilon_2\cap \Omega}
\partial_t{u}_\varepsilon\cdot\phi=\int_{-s}^0\int_{\Omega^\varepsilon_2\cap \Omega}
\operatorname{div}(A^\varepsilon\nabla {u}_\varepsilon)\cdot\phi\\
=&-\int_{-s}^0\int_{\Omega^\varepsilon_2\cap \Omega} A^\varepsilon\nabla
{u}_\varepsilon\cdot\nabla\phi+\int_{-s}^0\int_{\partial(\Omega^\varepsilon_2\cap \Omega)}A^\varepsilon\nabla {u}_\varepsilon\cdot n_\varepsilon\cdot\phi,
\end{aligned}\end{equation}
where $n_\varepsilon$ denotes the outward unit normal to $\partial(\Omega^\varepsilon_2 \cap \Omega)$. Now we need to determine the trace of $A^\varepsilon\nabla {u}_\varepsilon\cdot n_\varepsilon$ on the boundary $\partial(\Omega^\varepsilon_2 \cap \Omega)$ and bound the second term uniformly in $\varepsilon$ in \eqref{2.7}.

For this purpose, we note that, for almost every $t\in[-4,0]$, we have $u_\varepsilon(\cdot,t)\in
H^1(\Omega^\varepsilon_2 \cap \Omega)$. Now, by viewing $t$ as a parameter, we introduce the following equation:
\begin{equation} \label{2.8}
\left\{\begin{aligned}
- \operatorname{div}\left(A^\varepsilon \nabla v_\varepsilon(\cdot,t) \right) &= 0 \quad\quad \ \quad\text{ in } \Omega^\varepsilon_2 \cap \Omega,\\
v_\varepsilon(\cdot,t)&=u_\varepsilon(\cdot,t) \quad \text{ on } \partial(\Omega^\varepsilon_2 \cap \Omega).\\
\end{aligned}\right.
\end{equation}

Note that $u_\varepsilon(\cdot, t)=0$ for $x\in \partial\Omega^\varepsilon \cap \Omega$, then we set
\begin{equation}\label{2.9}
\tilde{v}_\varepsilon(x,t)=\left\{\begin{aligned}
&v_\varepsilon(x,t),\quad  &\text{ if }x\in \Omega^\varepsilon_2 \cap \Omega,\\
&0, &\text{ if }x\in \Omega\setminus\Omega^\varepsilon_2.
\end{aligned}\right.\end{equation}

Firstly, basic energy estimates yield that
\begin{equation}\label{2.10}\begin{aligned}
&||\nabla \tilde{v}_\varepsilon(\cdot,t)||_{L^2(\Omega)}\leq ||\nabla v_\varepsilon(\cdot,t)||_{L^2(\Omega^\varepsilon_2 \cap \Omega)}\\[6pt]
\leq& C||\nabla u_\varepsilon(\cdot,t)||_{L^2(\Omega^\varepsilon \cap \Omega)}=C||\nabla u_\varepsilon(\cdot,t)||_{L^2( \Omega)}.
\end{aligned}\end{equation}

Secondly, according to \cite[Theorem 4.15]{MR4242224} (see also \cite[Lemma 20.2]{MR2328004} for the
case of $p=2$), we know that the trace of $A^\varepsilon \nabla v_\varepsilon(\cdot,t)\cdot
n_\varepsilon$ on $\partial(\Omega^\varepsilon_2 \cap B_2)$ is meaningful and belongs to
$H^{-1/2}(\partial(\Omega^\varepsilon_2 \cap B_2))$.  Note that this
$H^{-1/2}(\partial(\Omega^\varepsilon_2 \cap B_2))$ bound of $A^\varepsilon \nabla
v_\varepsilon(\cdot,t)\cdot n_\varepsilon$ given in \cite[Theorem 4.15]{MR4242224} may depend on $\varepsilon$, therefore we need to bound this trace uniformly.

Now, for the same $\phi$ in \eqref{2.7}, we have $\phi(\cdot,t)\in H^1_0(\Omega)$ and a direct computation shows that
\begin{equation}\label{2.11}\begin{aligned}
&\int_{\Omega}- \operatorname{div}\left(A^\varepsilon \nabla \tilde{v}_\varepsilon(\cdot,t) \right)\cdot \phi(\cdot,t)=\int_{\Omega}A^\varepsilon \nabla \tilde{v}_\varepsilon(\cdot,t)\nabla \phi(\cdot,t)\\
=&\int_{\Omega^\varepsilon_2 \cap \Omega}A^\varepsilon \nabla {v}_\varepsilon(\cdot,t)\nabla \phi(\cdot,t)\\
=&\int_{\Omega^\varepsilon_2 \cap \Omega}- \operatorname{div}\left(A^\varepsilon \nabla {v}_\varepsilon(\cdot,t) \right)\cdot \phi(\cdot,t)
+\int_{\partial(\Omega^\varepsilon_2 \cap \Omega)}A^\varepsilon\nabla {v}_\varepsilon(\cdot,t) \cdot n_\varepsilon \cdot \phi(\cdot,t)\\
=&\int_{\partial(\Omega^\varepsilon_2 \cap \Omega)}A^\varepsilon\nabla {u}_\varepsilon(\cdot,t) \cdot n_\varepsilon \cdot \phi(\cdot,t).
\end{aligned}\end{equation}

Therefore, it directly follows from \eqref{2.10} and \eqref{2.11} after integration with respect to $t$ over $(-s,0)$ that
\begin{equation}\label{2.12}\begin{aligned}
\left|\int_{-s}^0\int_{\partial(\Omega^\varepsilon_2 \cap \Omega)}A^\varepsilon\nabla {u}_\varepsilon \cdot n_\varepsilon \cdot \phi\right|
=&\left|\int_{-s}^0\int_{\Omega}\operatorname{div}\left(A^\varepsilon \nabla \tilde{v}_\varepsilon \right)\cdot \phi\right|\\[9pt]
\leq& C||\nabla \tilde{v}_\varepsilon||_{L^2(\Omega\times(-s,0))}||\nabla \phi||_{L^2(\Omega\times(-s,0))}\\[9pt]
\leq& C||\nabla u_\varepsilon||_{L^2(\Omega\times(-s,0))}||\nabla \phi||_{L^2(\Omega\times(-s,0))}.\\[9pt]
\end{aligned}\end{equation}

Now, back to \eqref{2.7} after in view of \eqref{2.12}, for any $\phi\in C^\infty(-s,0;C^\infty_0(\Omega))$, there holds
\begin{equation*}
\left|\int_{-s}^0\int_{\Omega}\partial_t \tilde{u}_\varepsilon\cdot \phi\right|\leq C||\nabla u_\varepsilon||_{L^2(\Omega\times(-s,0)}||\nabla \phi||_{L^2(\Omega\times(-s,0))},
\end{equation*}
which, by density, immediately implies the desired estimates \eqref{2.6}.
\end{proof}

\begin{rmk}\label{r2.3}
(i) If we choose $\Omega=\Omega^\varepsilon \cap B_r$ in \eqref{2.6}, then
\begin{equation*}
||\partial_t \tilde{u}_\varepsilon||_{L^2(-s,0;H^{-1}(\Omega^\varepsilon_r))}\leq C ||\nabla u_\varepsilon||_{L^2(\Omega^\varepsilon_r\times(-s,0))}.
\end{equation*}

%Moreover, choosing $\Omega=\Omega^\varepsilon$ in \eqref{2.6} yields that
%\begin{equation*}
%||\partial_t \tilde{u}_\varepsilon||_{L^2(-s,0;H^{-1}(\Omega^\varepsilon))}\leq C ||\nabla u_\varepsilon||_{L^2(\Omega^\varepsilon \times(-s,0))}.
%\end{equation*}

(ii) For more general domain $\Omega$ and $\Omega^\varepsilon$, the estimates \eqref{2.10} continue to hold under the same extension of $\tilde{v}_\varepsilon$. Now, a careful detection shows that if the term $\int_{\partial(\Omega^\varepsilon_2 \cap \Omega)}A^\varepsilon\nabla {v}_\varepsilon(\cdot,t) \cdot n_\varepsilon \cdot \phi(\cdot,t)$ in \eqref{2.7} and \eqref{2.11} is meaningful, then the estimate \eqref{2.12} continues to hold. Actually, this problem is related to the domains in which the divergence theorem applies, and this question need to be dived into geometric measure theory a bit. We do dot pursue the details and refer the readers to \cite[Section 3.3]{MR1857292} for more results about this issue.

\end{rmk}

\section{Large-scale Meyers estimate}\label{l3}
In order to obtain a large-scale Meyers estimate, we first introduce the following Caccioppoli inequality of parabolic version, whose proof can be found in \cite[Lemma B.3]{MR3812862} for the non-oscillating domain.

\begin{lemma}[parabolic Caccioppoli inequality]\label{l3.1}
Let $\varepsilon^*=\varepsilon\zeta(\varepsilon,1)$ and $r\in (\varepsilon^*,1)$.
Suppose that $u_\varepsilon\in L^2(-4r^2,0; H^1(\Omega^\varepsilon_{2r}))$  satisfies
\begin{equation} \label{3.1}
\left\{\begin{aligned}
\partial_t u_\varepsilon - \operatorname{div}\left(A^\varepsilon \nabla u_\varepsilon \right) = 0 \quad \text{ in } Q_{2r}^\varepsilon,\\
u_\varepsilon=0 \quad \text{ on } \Delta_{2r}^\varepsilon,\\
\end{aligned}\right.
\end{equation}
then there exists a constant $C$ depending only on $d, \Lambda$ and $m$, such that
\begin{equation}
\label{3.2}
 \left\| \nabla u_\varepsilon \right\|_{L^2(Q_r)}
\leq Cr^{-1}\left\|  u_\varepsilon \right\|_{L^2(Q_{2r})}
\end{equation}
and
\begin{equation}
\label{3.3}
\sup_{s\in (-r^2,0)} \left\|  u_\varepsilon(\cdot,s) \right\|_{L^2(B_r)}
\leq C\left\|  \nabla u_\varepsilon \right\|_{L^2(Q_{2r})}.
\end{equation}
\end{lemma}
\begin{proof}
We take $\eta_r \in C^\infty(Q_{2r})$ satisfying
\begin{equation*} \label{}
0\leq \eta_r \leq 1, \
\eta_r \equiv 1 \ \mbox{on} \ Q_r,  \ \eta \equiv 0 \ \mbox{on} \ Q_{2r}\setminus Q_{3r/2},\
\left| \partial_t \eta \right| + \left| \nabla \eta \right|^2 \leq Cr^{-2}.
\end{equation*}

Note that the function $\phi := \eta^2_r u_\varepsilon \in L^2(I_{2r};H^{1}_0(\Omega^\varepsilon_{2r}))$, then testing the equation \eqref{3.1} with $\phi$ implies that
\begin{equation} \label{3.4}
\int_{Q_{2r}^\varepsilon}- \phi \cdot\partial_tu_\varepsilon
= \int_{Q_{2r}^\varepsilon}\nabla \phi\cdot A^\varepsilon\nabla u_\varepsilon.
\end{equation}

A direct computation shows that
\begin{equation}\begin{aligned} \label{3.5}
\int_{Q_{2r}^\varepsilon}\nabla \phi\cdot A^\varepsilon \nabla u_\varepsilon
&
\geq \frac{1}{\Lambda} \int_{Q_{2r}^\varepsilon} \eta_r^2 \left| \nabla u_\varepsilon \right|^2 -  C \int_{Q_{2r}^\varepsilon} \eta_r \left| \nabla \eta_r \right| |u_\varepsilon| \left| \nabla u_\varepsilon \right|
\\ &
\geq \frac{1}{2\Lambda} \int_{Q_{2r}^\varepsilon} \eta_r^2 \left| \nabla u_\varepsilon \right|^2 - C\int_{Q_{2r}^\varepsilon} \left| \nabla \eta_r \right|^2 |u_\varepsilon|^2
\\ &
\geq  \frac{1}{2\Lambda} \int_{Q_{2r}^\varepsilon} \eta_r^2 \left| \nabla u_\varepsilon \right|^2 - Cr^{-2} \int_{Q_{2r}^\varepsilon}  |u_\varepsilon|^2
\end{aligned}\end{equation}

and
\begin{equation}\label{3.6}\begin{aligned}
\int_{Q_{2r}^\varepsilon} -\eta^2_r u_\varepsilon \cdot\partial_tu_\varepsilon
&
\leq - \int_{Q_{2r}^\varepsilon} \partial_t \left( \frac12 \eta^2_r u_\varepsilon^2 \right) + \int_{Q_{2r}^\varepsilon} \eta_r \left| \partial_t\eta_r \right| u_\varepsilon^2\\
 & \leq - \frac12 \int_{\Omega_{2r}^\varepsilon} \eta_r^2(x,0) u_\varepsilon^2(x,0)\,dx + Cr^{-2} \int_{Q_{2r}^\varepsilon} u_\varepsilon^2.
\end{aligned}\end{equation}

Combining \eqref{3.4}-\eqref{3.6} yields that
\begin{equation*} \label{}
\frac12 \int_{\Omega^\varepsilon_{2r}} \eta_r^2(x,0) u_\varepsilon^2(x,0)\,dx +  \frac1{2\Lambda} \int_{Q_{2r}^\varepsilon} \eta_r^2 \left| \nabla u_\varepsilon \right|^2
\leq
Cr^{-2} \int_{Q_{2r}^\varepsilon}  |u_\varepsilon|^2,
\end{equation*}
which eventually implies the desired estimate \eqref{3.2}.

To proceed, testing the equation \eqref{3.1} with $\phi= \eta^2_r u_\varepsilon$ and integrating the resulting equation over $\Omega_{2r}^\varepsilon\times(-4r^2,s)$ for fixed $s\in (-4r^2,0)$ yields that

\begin{equation}\label{3.7}
\int_{-4r^2}^s\int_{\Omega_{2r}^\varepsilon}- \phi \cdot\partial_tu_\varepsilon
= \int_{-4r^2}^s\int_{\Omega_{2r}^\varepsilon}\nabla \phi\cdot A^\varepsilon\nabla u_\varepsilon.
\end{equation}

Now, a direct computation yields that
\begin{equation}\label{3.8}
\begin{aligned}
&\int_{-4r^2}^s\int_{\Omega_{2r}^\varepsilon}\nabla \phi\cdot A^\varepsilon \nabla u_\varepsilon\\[5pt]
\geq& -C \left\| \nabla \eta_r \nabla u_\varepsilon \right\|_{L^2({\Omega_{2r}^\varepsilon}\times(-4r^2,s))}  \left\| u_\varepsilon \eta_r \right\|_{L^2({\Omega_{2r}^\varepsilon}\times(-4r^2,s))}
\\[5pt]
\geq &-C  \left\| \nabla u_\varepsilon \right\|_{L^2({\Omega_{2r}^\varepsilon}\times(-4r^2,s))}^2 - \frac1{16}r^{-2}  \int_{-4r^2}^s\int_{\Omega_{2r}^\varepsilon} \eta_r^2 u_\varepsilon^2
\\[5pt]
\geq& -C  \left\|  \nabla u_\varepsilon \right\|_{L^2({\Omega_{2r}^\varepsilon}\times(-4r^2,s))}^2 - \frac14 \sup_{t\in (-4r^2,0)} \int_{\Omega_{2r}^\varepsilon } \eta_r^2(x,t) u_\varepsilon^2(x,t)\,dx,
\end{aligned}\end{equation}
and
\begin{equation}\label{3.9}
\begin{aligned}
\int_{-4r^2}^s\int_{\Omega_{2r}^\varepsilon} -\eta^2_r u_\varepsilon \cdot\partial_tu_\varepsilon
&
\leq - \int_{-4r^2}^s\int_{\Omega_{2r}^\varepsilon} \partial_t \left( \frac12 \eta^2_r u_\varepsilon^2 \right) + \int_{-4r^2}^s\int_{\Omega_{2r}^\varepsilon} \eta_r \left| \partial_t\eta_r \right| u_\varepsilon^2\\
 & \leq - \frac12 \int_{\Omega_{2r}^\varepsilon} \eta_r^2(x,s) u_\varepsilon^2(x,s)\,dx + Cr^{-2} \int_{Q_{2r}^\varepsilon} u_\varepsilon^2\\
 & \leq - \frac12 \int_{\Omega_{2r}^\varepsilon} \eta_r^2(x,s) u_\varepsilon^2(x,s)\,dx + C\left\|  \nabla u_\varepsilon \right\|_{L^2({Q_{2r}^\varepsilon})}^2,
\end{aligned}\end{equation}
we have used the Poincar\'e inequality in \eqref{3.9}.

Consequently, combining \eqref{3.7}-\eqref{3.9} yields that
\begin{multline*} \label{}
\frac12 \int_{\Omega^\varepsilon_{2r}} \eta_r^2(x,s) u_\varepsilon^2(x,s)\,dx
\leq C\left\|  \nabla u_\varepsilon \right\|_{L^2({Q_{2r}^\varepsilon})}^2 +  \frac14 \sup_{t\in (-4r^2,0)}  \int_{\Omega^\varepsilon \cap B_{2r}} \eta_r^2(t,x) u^2(t,x)\,dx.
\end{multline*}
Taking the supremum over $s\in (-4r^2,0)$ and rearranging finally imply the desired estimate \eqref{3.3}.
\end{proof}

The following lemma states a reverse H\"{o}lder inequality for $u_\varepsilon$, whose proof can be found in \cite[Lemma B.4]{MR3812862} for the non-oscillating domain.

\begin{lemma}\label{l3.2}
Let $\varepsilon^*=\varepsilon\zeta(\varepsilon,1)$ and $r\in (\varepsilon^*,1)$.
Suppose that $u_\varepsilon\in L^2(-16r^2,0; H^1(\Omega^\varepsilon_{4r}))$  satisfies
\begin{equation} \label{3.10}
\left\{\begin{aligned}
\partial_t u_\varepsilon - \operatorname{div}\left(A^\varepsilon \nabla u_\varepsilon \right) = 0 \quad \text{ in } Q_{4r}^\varepsilon,\\
u_\varepsilon=0 \quad \text{ on } \Delta_{4r}^\varepsilon.\\
\end{aligned}\right.
\end{equation}
For simplicity,
denote $q:=2_* = 2d/(2+d)$ for $d\geq 3$ and $q>1$ for d=2.
Then there exists a constant $C<\infty$, depending only on $d,\Lambda,m$, such that for every $\delta>0$,
\begin{equation}\label{3.11}\begin{aligned}
\left(\fint_{Q_r}|\nabla u_\varepsilon|^2\right)^{1/2}\leq  C\delta^{-3}\left(\fint_{Q_{2r}}|\nabla
u_\varepsilon|^q\right)^{1/q}+\delta\left(\fint_{Q_{4r}}|\nabla u_\varepsilon|^2\right)^{1/2}.
\end{aligned}\end{equation}
\end{lemma}
\begin{proof}
According to \eqref{3.3}, we have
\begin{equation}\label{3.12}\begin{aligned}
\int_{Q_{2r}^\varepsilon} \left| u_\varepsilon  \right|^2&
\leq  \left(\sup_{t\in I_{2r}} \int_{\Omega^\varepsilon_{2r}} \left| u_\varepsilon(x,t)  \right|^2 \,dx \right)^{\frac12} \int_{I_{2r}} \left( \int_{\Omega^\varepsilon_{2r}} \left| u_\varepsilon(x,t)  \right|^2 \,dx\right)^{\frac12} \,dt
\\[5pt] &
\leq C  \left\|  \nabla u_\varepsilon \right\|_{L^2(Q_{4r}^\varepsilon)}
\int_{I_{2r}} \left( \int_{\Omega^\varepsilon_{2r}} \left| u_\varepsilon(x,t)  \right|^2 \,dx\right)^{\frac12}dt.
\end{aligned}\end{equation}
Let $q'$ be the H\"older conjugate exponent to $q$, i.e., $q' = 2^*=\frac{2d}{d-2}$ for $d>2$ and $q'<\infty$ for $d=2$. Due to the H\"older and Sobolev-Poincar\'e inequalities, there holds
\begin{equation}\label{3.13}\begin{aligned}
&\int_{I_{2r}} \left( \int_{\Omega^\varepsilon_{2r}} \left| u_\varepsilon(x,t)  \right|^2 \,dx\right)^{\frac12} dt
= \int_{I_{2r}} \left( \int_{ B_{2r}} \left| u_\varepsilon(x,t)  \right|^2 \,dx\right)^{\frac12} dt &
\\
\leq& \int_{I_{2r}} \left( \int_{ B_{2r}} \left| u_\varepsilon(x,t)\right|^q \,dx \right)^{\frac{1}{2q}} \left(\int_{ B_{2r} } \left| u_\varepsilon(x,t)\right|^{q'}\,dx \right)^{\frac{1}{2q'}}\,dt
\\
\leq& C r^{1 + d\left( \frac14-\frac1{2q}\right) }
\int_{I_{2r}}
\left( \int_{ B_{2r}} \left| \nabla u_\varepsilon(x,t)\right|^q \,dx \right)^{\frac{1}{2q}}  \left(\int_{ B_{2r} } \left| \nabla u_\varepsilon(x,t)\right|^{2}\,dx \right)^{\frac{1}{4}}\,dt
\\
\leq& C r^{1 + d\left( \frac14-\frac1{2q}\right) }
\left\| \nabla u_\varepsilon \right\|_{L^q(Q_{2r})}^{\frac12} \left( \int_{I_{2r}} \left( \int_{ B_{2r}} \left| \nabla u_\varepsilon(x,t) \right|^2 \,dx \right)^{\frac{(2q)'}{4}} \,dt  \right)^{\frac{1}{(2q)'}},
\end{aligned}\end{equation}
with $\frac1{2q}+\frac1{(2q)'}=1$.

Due to H\"older's inequality in time after noting $\frac{(2q)'}{4}=\frac q{2(2q-1)}<\frac12$, there holds

\begin{equation}\label{3.14}\begin{aligned}
\left( \int_{I_{2r}} \left( \int_{ B_{2r}} \left| \nabla u_\varepsilon(x,t) \right|^2 \,dx \right)^{\frac{(2q)'}{4}} \,dt  \right)^{\frac{1}{(2q)'}}
\leq
Cr^{\frac{2}{(2q)'} -\frac12} \left(    \int_{I_{2r}}  \int_{ B_{2r}}  \left| \nabla u_\varepsilon (x,t) \right|^2  \,dx \,dt  \right)^{\frac14}.
\end{aligned}\end{equation}

Denote $\kappa:=d\left(\frac14-\frac1{2q}\right) + \frac{2}{(2q)'} +\frac12=(d+2)\left(\frac14-\frac1{2q}\right)+2$. Combining \eqref{3.12}-\eqref{3.14} yields that
\begin{equation*}
\left\| u_\varepsilon \right\|_{L^2({Q_{2r}})}^2
\leq Cr^{\kappa}
 \left\| \nabla u_\varepsilon \right\|_{L^q({Q_{2r}})}^{\frac12} \left\|  \nabla u_\varepsilon \right\|_{L^2(Q_{4r})} ^{\frac32}.
\end{equation*}

Normalizing the norms and using \eqref{3.2}, we have
\begin{equation*}\begin{aligned}
\fint_{Q_r}|\nabla u_\varepsilon|^2\leq C \left(\fint_{Q_{2r}}|\nabla u_\varepsilon|^q\right)^{1/(2q)}\left(\fint_{Q_{4r}}|\nabla u_\varepsilon|^2\right)^{3/4}.
\end{aligned}\end{equation*}

\noindent
Consequently, due to Young's inequality, there holds, for every $\delta >0$,

\begin{equation*}\begin{aligned}
\left(\fint_{Q_r}|\nabla u_\varepsilon|^2\right)^{1/2}\leq  C\delta^{-3}\left(\fint_{Q_{2r}}|\nabla u_\varepsilon|^q\right)^{1/q}+\delta\left(\fint_{Q_{4r}}|\nabla u_\varepsilon|^2\right)^{1/2},
\end{aligned}\end{equation*}
which is the desired estimate \eqref{3.11}.
\end{proof}

After obtaining the reverse H\"older inequality, we are ready to state the following large-scale Meyers estimate for $u_\varepsilon$.
\begin{lemma}[Large-scale Meyers estimate]\label{l3.3}
Fix $r_1\in (\varepsilon^*,1)$, then for any $r\in (0,1/9)$, there exists a constant $p_0>2$, depending only on $\Lambda,m,d$,  such that
\begin{equation}\label{3.15}\begin{aligned}
\left(\fint_{Q_r}|\mathcal{M}_{r_1}^2[\nabla u_\varepsilon]|^{p_0}\right)^{1/p_0}\leq  C\left(\fint_{Q_{9r}}|\mathcal{M}_{r_1}^2[\nabla u_\varepsilon]|^2\right)^{1/2}.
\end{aligned}\end{equation}
\end{lemma}

\begin{proof}
We first assume $r\ge r_1$. Then by \eqref{2.1}, \eqref{2.3} and \eqref{3.11}, a direct computation shows that
		\begin{equation}\label{3.16}
			\begin{aligned}
				\bigg( \fint_{Q_{r}} |\mathcal{M}_{r_1}^2 [\nabla u_\varepsilon] |^{2} \bigg)^{1/2}
				&\le C \bigg( \fint_{Q_{2r}} | \nabla u_\varepsilon|^2  \bigg)^{1/2} \\[5pt]
				& \le C\delta \bigg( \fint_{Q_{8r}} |\nabla u_\varepsilon|^2 \bigg)^{1/2} + C\delta^{-3}\bigg( \fint_{Q_{4r}} |\nabla u_\varepsilon|^{q} \bigg)^{1/q}  \\[5pt]
				& \le C \delta \bigg( \fint_{Q_{9r}} |\mathcal{M}_{r_1}^2[ \nabla u_\varepsilon]|^2 \bigg)^{1/2}+ C\delta^{-3}\bigg( \fint_{Q_{9r}} | \mathcal{M}_{r_1}^{2} [\nabla u_\varepsilon]|^{q}  \bigg)^{1/q},
			\end{aligned}
		\end{equation}
where the constant $2>q>1$ is given in Lemma \ref{l3.2} and $\delta\in (0,1)$ is arbitraty.

		For $0<r<r_1$, it follows from \eqref{2.5} that
		\begin{equation*}
			\|  \mathcal{M}_{r_1}^2 [\nabla u_\varepsilon] \|_{L^\infty(Q_r)} \le C\fint_{Q_{4r}} \mathcal{M}_{r_1}^2 [\nabla u_\varepsilon].
		\end{equation*}
		The two inequalities above imply that a weaker reverse H\"{o}lder inequality holds for all scales $r\in (0,1)$ with $Q_{9r+r_1} \subset Q_1$. By the Gehring's inequality \cite[Lemma B.5]{MR3812862} of parabolic version after choosing $\delta$ sufficiently small, there exists some $\sigma > 0$ such that for all $r\in (0,1)$ with $Q_{9r+r_1} \subset Q_1$,
		\begin{equation*}
			\bigg( \fint_{Q_{r}} |\mathcal{M}_{r_1}^2 [\nabla u_\varepsilon] |^{p_0} \bigg)^{1/p_0} \le C\bigg( \fint_{Q_{9r}} |\mathcal{M}_{r_1}^2[ \nabla u_\varepsilon]|^2 \bigg)^{1/2},
		\end{equation*}
	which is the desired estimate \eqref{3.15}.
		%Finally, a covering argument can be used to adjust the size of cubes at the price of a larger constant $C$ than in \eqref{3.17}.

\end{proof}

To proceed, for $\varepsilon \leq r $, we denote
\begin{equation}\label{3.17}\begin{aligned}
T^{\varepsilon,+}_r=:B_r \cap \{x\in\mathbb{R}^d:x\cdot n_r<r \zeta(r,\varepsilon/r)\},\\
\partial T^{\varepsilon,+}_r=:B_r \cap \{x\in\mathbb{R}^d:x\cdot n_r=r \zeta(r,\varepsilon/r)\}\end{aligned}\end{equation}
and
\begin{equation}\label{3.18}\begin{aligned}
T^{\varepsilon,-}_r=:B_r \cap \{x\in\mathbb{R}^d:x\cdot n_r<-r \zeta(r,\varepsilon/r)\},\\
\partial T^{\varepsilon,-}_r=:B_r \cap \{x\in\mathbb{R}^d:x\cdot n_r=-r \zeta(r,\varepsilon/r)\},\end{aligned}\end{equation}
where $n_r\in \mathbb{S}^{d-1}$ is a unit ``outer normal" vector such that $T^{\varepsilon,-}_r\subset \Omega_r^\varepsilon\subset T^{\varepsilon,+}_r$ with $|T^{\varepsilon,+}_r/T^{\varepsilon,-}_r|\leq C_d r^d\zeta(r,\varepsilon/r)=C_d \zeta(r,\varepsilon/r)|\Omega_r^\varepsilon|$, given in Definition \ref{d1.1}.

Note that $T^{\varepsilon,\pm}_r$ are Lipschitz bounded domains with the Lipschitz character independent of $\varepsilon$ and $r$.

Now, we introduce the following large-scale regularity, which plays an important role in obtaining the excess estimates in Section \ref{l4}.
\begin{lemma}\label{l3.4}
Let $\varepsilon^*=\varepsilon\zeta(\varepsilon,1)$ and $r_2,r_1\in (\varepsilon^*,1/9)$ with $r_2\geq r_1$.
Suppose that $u_\varepsilon\in L^2(-81r_2^2,0; H^1(\Omega^\varepsilon_{9r_2}))$  satisfies the equation \eqref{1.1}. Now let $w_\varepsilon=w_\varepsilon^{9r_2}$ (we will drop the superscript $9r_2$ for simplicity, if there is no ambiguity) be the weak solution of
\begin{equation} \label{3.19}
\left\{\begin{aligned}
\partial_t w_\varepsilon - \operatorname{div}\left(A^\varepsilon \nabla w_\varepsilon \right)= 0& \quad \text{ in } T^{\varepsilon,+}_{9r_2}\times I_{9r_2},\\
w_\varepsilon=u_\varepsilon& \quad\text{ on } \partial _p \left(T^{\varepsilon,+}_{9r_2}\times I_{9r_2}\right).\\
\end{aligned}\right.
\end{equation}
%Moreover, similar as $u_\varepsilon$, we extend $w_\varepsilon$ across the boundary by zero-extension, still denoted by $w_\varepsilon$. 
Then there exists a constant $p_1>2$, depending only on $m,d,\Lambda$, such that
\begin{equation}\label{3.20}
			\bigg( \fint_{T^{\varepsilon,+}_{9r_2}\times I_{9r_2}} | \mathcal{M}_{r_1}^2 [\nabla w_\varepsilon]|^{p_1}  \bigg)^{1/p_1}
			 \leq C\bigg( \fint_{T^{\varepsilon,+}_{9r_2}\times I_{9r_2}} | \mathcal{M}_{r_1}^2 [\nabla u_\varepsilon]|^2  \bigg)^{1/2}.
	\end{equation}
\end{lemma}
\begin{proof}We extend $w_\varepsilon$ to be $u_\varepsilon$ across the parabolic boundary $\partial _p \left(T^{\varepsilon,+}_{9r_2}\times I_{9r_2}\right)$.
First, according to \eqref{1.6}, we know that
\begin{equation}\label{3.21}\begin{aligned}
u_\varepsilon\in L^2(I_{9r_2};H^{1}(B_{9r_2}))&,\quad \partial_tu_\varepsilon\in L^2(I_{9r_2};H^{-1}(B_{T^{\varepsilon,+}_{9r_2}})),\\[5pt]
\mathcal{M}_{r_1}^2[\nabla u_\varepsilon]\in L^{2}(Q_{9r_2}),\quad & ||\partial_tu_\varepsilon||_{L^2(I_{9r_2};H^{-1}(T^{\varepsilon,+}_{9r_2}))}\leq C||\nabla u_\varepsilon||_{L^2(T^{\varepsilon,+}_{9r_2}\times I_{9r_2})}.
\end{aligned}\end{equation}

Recall that the extension of $u_\varepsilon$ across the rough boundary, then, basic energy estimates yield that

\begin{equation}\label{3.22}
\int_{Q_{9r_2}}|\nabla w_\varepsilon|^2=\int_{{T^{\varepsilon,+}_{9r_2}\times I_{9r_2}}}|\nabla w_\varepsilon|^2 \leq C \int_{T^{\varepsilon,+}_{9r_2}\times I_{9r_2}}|\nabla u_\varepsilon|^2=C\int_{Q_{9r_2}}|\nabla u_\varepsilon|^2.
\end{equation}

To proceed, for any $r\geq r_1$ and any $\alpha>0$, by the reverse H\"{o}lder inequality of parabolic version \cite[Lemma B.7]{MR3812862} for the non-oscillating boundary, we have
\begin{equation}\label{3.23}\begin{aligned}
\left(\fint_{Q_r}|\nabla w_\varepsilon-\nabla u_\varepsilon|^2\right)^{1/2}\leq &\ \frac C\delta\left(\fint_{Q_{4r}}|\nabla w_\varepsilon-\nabla u_\varepsilon|^q\right)^{1/q}+\delta\left(\fint_{Q_{4r}}|\nabla w_\varepsilon-\nabla u_\varepsilon|^2\right)^{1/2}\\
&\quad+ \delta\left(\fint_{Q_{4r}}|\nabla u_\varepsilon|^2\right)^{1/2},
\end{aligned}\end{equation}
with the same $q<2$ defined in Lemma \ref{l3.2}. Now, as the same computation in \eqref{3.16}, for any $r\geq r_1$, we have
\begin{equation}\label{3.24}
			\begin{aligned}
				\bigg( \fint_{Q_{r}} |\mathcal{M}_{r_1}^2 [\nabla w_\varepsilon-\nabla u_\varepsilon] |^{2}& \bigg)^{1/2}
				 \le {C}{ \delta} \bigg( \fint_{Q_{9r}} |\mathcal{M}_{r_1}^2[ \nabla w_\varepsilon-\nabla u_\varepsilon]|^2 \bigg)^{1/2}\\
+\frac C\delta&\bigg( \fint_{Q_{9r}} | \mathcal{M}_{r_1}^{2} [\nabla w_\varepsilon-\nabla u_\varepsilon]|^{q}  \bigg)^{1/q}+\delta \bigg( \fint_{Q_{9r}} |\mathcal{M}_{r_1}^2[\nabla u_\varepsilon]|^2 \bigg)^{1/2}.
			\end{aligned}
		\end{equation}
Meanwhile, for $r\leq r_1$, it follows from \eqref{2.5} that
		\begin{equation*}
			\|  \mathcal{M}_{r_1}^2 [\nabla w_\varepsilon-\nabla u_\varepsilon] \|_{L^\infty(Q_r)} \le C\fint_{Q_{4r}} \mathcal{M}_{r_1}^2 [\nabla w_\varepsilon-\nabla u_\varepsilon].
		\end{equation*}
Therefore, according to the two inequality above and the Gehring's inequality \cite[Lemma B.5]{MR3812862} of parabolic version after choosing $\delta$ sufficiently small, there exists $ p_0\geq p_1>2$, such that
\begin{equation}\label{3.25}
			\begin{aligned}
&\bigg( \fint_{Q_{r}} |\mathcal{M}_{r_1}^2 [\nabla w_\varepsilon-\nabla u_\varepsilon] |^{p_1} \bigg)^{1/p_1}\\
\le& C\bigg( \fint_{Q_{9r}} |\mathcal{M}_{r_1}^2[ \nabla w_\varepsilon-\nabla u_\varepsilon]|^2 \bigg)^{1/2}+
C \bigg( \fint_{Q_{9r}} |\mathcal{M}_{r_1}^2[\nabla u_\varepsilon]|^{p_1} \bigg)^{1/p_1}\\
\leq& C\bigg( \fint_{Q_{9r}} |\mathcal{M}_{r_1}^2[ \nabla w_\varepsilon-\nabla u_\varepsilon]|^2 \bigg)^{1/2}+
C \bigg( \fint_{Q_{81r}} |\mathcal{M}_{r_1}^2[\nabla u_\varepsilon]|^{2} \bigg)^{1/2},
			\end{aligned}
		\end{equation}
where we have used the Lemma \ref{l3.3} in the inequality. Note that the inequality above implies that we have obtained desired estimates near the rough boundary with $r\leq r_2/9$.
%which, after noting \eqref{3.21} and \eqref{3.22}, implies that
%\begin{equation}\label{3.26}
%\bigg( \fint_{Q_{r}} |\mathcal{M}_{r_1}^2 [\nabla w_\varepsilon] |^{p_1} \bigg)^{1/p_1}
%				 \le C \bigg( \fint_{Q_{10r}} |\mathcal{M}_{r_1}^2[\nabla u_\varepsilon]|^{2} \bigg)^{1/2}.
%\end{equation}

To proceed, using the global version of the reverse H\"older inequality for the non-oscillating  domain \cite[Lemma B.7]{MR3812862}, the interior version of the reverse H\"older inequality \cite[Lemma B.4]{MR3812862} and the idea of the computation in \eqref{3.16} yields the desired estimate \eqref{3.20} after a combination of the estimate \eqref{3.25}.
\end{proof}

Now, for the $w_\varepsilon$ defined in \eqref{3.19}, we have the following excess decay estimate.

\begin{lemma}\label{l3.5}
For every $r_2\in (\varepsilon,1/10)$, there holds
\begin{equation}\label{3.27}
\left(\fint_{Q_{r_2}}|\nabla u_\varepsilon-\nabla w_\varepsilon^{r_2}|^2\right)^{1/2}\leq C\varsigma (r_2,\varepsilon/r_2)^{\gamma}\left(\fint_{Q_{10r_2}}|\nabla u_\varepsilon|^2\right)^{1/2},
\end{equation}
where $\gamma=1/2-1/{p_1}$ with $p_1>2$ defined in Lemma \ref{l3.4}, and $w_\varepsilon^{r_2}$ satisfies the equation \eqref{3.19} with $9r_2$ replaced by $r_2$.
\end{lemma}
\begin{proof}
One can find the similar proof for the elliptic case in \cite{MR4199276}.
First, by \eqref{3.20} and \eqref{3.22}, we know that $w_\varepsilon^{r_2}-u_\varepsilon\in L^2(I_{r_2};H^1_0(T^{\varepsilon,+}_{r_2}))$. Then, testing the equation \eqref{3.19} by $w_\varepsilon^{r_2}-u_\varepsilon$ and integrating the resulting equation over $T^{\varepsilon,+}_{r_2}\times I_{r_2}$ yields that

\begin{equation}\label{3.28}
\int_{T^{\varepsilon,+}_{r_2}\times I_{r_2}}\partial_t w_\varepsilon^{r_2}\cdot (w_\varepsilon^{r_2}-u_\varepsilon)+\int_{T^{\varepsilon,+}_{r_2}\times I_{r_2}}A^\varepsilon \nabla w_\varepsilon^{r_2} \cdot \nabla(w_\varepsilon^{r_2}-u_\varepsilon)=0.
\end{equation}
Let $\phi=\phi(x)\in C^\infty_0(\mathbb{R}^d)$ be a smooth function such that $\phi=1$ on
$T^{\varepsilon,+}_{r_2}\setminus \ \Omega^\varepsilon_{r_2}$. Then
$(1-\phi)(w_\varepsilon^{r_2}-u_\varepsilon)\in L^2(I_{r_2};H^1_0(\Omega^\varepsilon_{r_2}))$.
Since $u_\varepsilon$ is a weak solution to the parabolic problem \eqref{1.1} in $Q_{r_2}^\varepsilon$, we have
\begin{equation*}
\int_{Q_{r_2}^\varepsilon}\partial_t u_\varepsilon\cdot (1-\phi)(w_\varepsilon^{r_2}-u_\varepsilon)+\int_{Q_{r_2}^\varepsilon}A^\varepsilon \nabla u_\varepsilon \cdot \nabla[(1-\phi)(w_\varepsilon^{r_2}-u_\varepsilon)]=0,
\end{equation*}
which is equivalent to
\begin{equation}\label{3.29}
\int_{T^{\varepsilon,+}_{r_2}\times I_{r_2}}\partial_t u_\varepsilon\cdot (1-\phi)(w_\varepsilon^{r_2}-u_\varepsilon)+\int_{T^{\varepsilon,+}_{r_2}\times I_{r_2}}A^\varepsilon \nabla u_\varepsilon \cdot \nabla[(1-\phi)(w_\varepsilon^{r_2}-u_\varepsilon)]=0,
\end{equation}
due to $\nabla u_\varepsilon=0$, $1-\phi=0$ for $(x,t)\in (T^{\varepsilon,+}_{r_2}\setminus \Omega^\varepsilon)\times I_{r_2}$ and $\partial_t u_\varepsilon\in L^2(I_{r_2};H^{-1}(\Omega^\varepsilon_{r_2})\cap H^{-1}(T^{\varepsilon,+}_{r_2}))$.
Combining \eqref{3.28}-\eqref{3.29} yields that
\begin{equation}\label{3.30}\begin{aligned}
\int_{T^{\varepsilon,+}_{r_2}\times I_{r_2}}\partial_t (w_\varepsilon^{r_2}-u_\varepsilon)\cdot
(w_\varepsilon^{r_2}-u_\varepsilon)+\int_{T^{\varepsilon,+}_{r_2}\times I_{r_2}}A^\varepsilon \nabla
(w_\varepsilon^{r_2}-u_\varepsilon) \cdot \nabla(w_\varepsilon^{r_2}-u_\varepsilon)&\\
=-\int_{T^{\varepsilon,+}_{r_2}\times I_{r_2}}\partial_t u_\varepsilon\cdot \phi
(w_\varepsilon^{r_2}-u_\varepsilon)-\int_{T^{\varepsilon,+}_{r_2}\times I_{r_2}}A^\varepsilon \nabla u_\varepsilon \cdot \nabla[\phi(w_\varepsilon^{r_2}-u_\varepsilon)]&.
\end{aligned}\end{equation}

To proceed, we need only to choose a suitable test function $\phi$. In view of
$T^{\varepsilon,+}_{r_2}\setminus \Omega_{r_2}^\varepsilon\subset T^{\varepsilon,+}_{r_2}\setminus
T^{\varepsilon,-}_{r_2}$ and the definitions of $T^{\varepsilon,+}_{r_2}$ and
$T^{\varepsilon,-}_{r_2}$ in \eqref{3.17}-\eqref{3.18}, we may choose $\phi$ such that $\phi=1$ on
$T^{\varepsilon,+}_{r_2}\setminus T^{\varepsilon,-}_{r_2}$ and $\phi=0$ in
$T^{\varepsilon,+}_{r_2}\cap \{x\in \mathbb{R}^d:x\cdot n_{r_2}<-4r_2\zeta(r_2,\varepsilon/r_2)\}$. Moreover, $|\nabla \phi|\leq C(r_2\varsigma(r_2,\varepsilon/r_2))^{-1}$.

Denote the set $T^{\varepsilon,+}_{r_2}\cap \{x\in \mathbb{R}^d:x\cdot n_{r_2}>-4r_2\zeta(r_2,\varepsilon/r_2)\}$ by $P_{r_2}$. Note that $P_{r_2}$ is a lamina-like region whose radius is $r_2$ and thickness is $5r_2\varsigma(r_2,\varepsilon/r_2)$. Thus $|P_{r_2}|\leq Cr_2^d\varsigma(r_2,\varepsilon/r_2)$. Moreover, note that $\text{supp}\left(\phi(w_\varepsilon^{r_2}-u_\varepsilon)\right)\subset P_{r_2}\times I_{r_2}$.

Now, turn back to the equality \eqref{3.30}. By the ellipticity condition \eqref{1.9}, we have
\begin{equation}\label{3.31}
\begin{aligned}
&\frac 1 2\int_{T^{\varepsilon,+}_{r_2}\times \{0\}}|\nabla w_\varepsilon^{r_2}-\nabla u_\varepsilon|^2+
\Lambda\int_{T^{\varepsilon,+}_{r_2}\times I_{r_2}}|\nabla w_\varepsilon^{r_2}-\nabla u_\varepsilon|^2\\[5pt]
\leq &||\partial_t u_\varepsilon||_{L^2(I_{r_2};H^{-1}(P_{r_2}))}\cdot \left(\int_{P_{r_2}\times I_{r_2}}|\nabla(\phi(w_\varepsilon^{r_2}-u_\varepsilon))|^2\right)^{1/2}\\
&\quad\quad\quad\quad\quad+\Lambda \left(\int_{P_{r_2}\times I_{r_2}}|\nabla u_\varepsilon|^2\right)^{1/2}\left(\int_{P_{r_2}\times I_{r_2}}|\nabla\left(\phi
(w_\varepsilon^{r_2}-u_\varepsilon)\right)|^2\right)^{1/2}\\[5pt]
\leq& C\left(\int_{P_{r_2}\times I_{r_2}}|\nabla u_\varepsilon|^2\right)^{1/2}\left(\int_{P_{r_2}\times I_{r_2}}|\nabla
(w_\varepsilon^{r_2}-u_\varepsilon)|^2\right)^{1/2},\\[5pt]
\end{aligned}\end{equation}
where in the inequality above, we have used \eqref{1.6} and the following Poincar\'e inequality
\begin{equation*}
\int_{P_{r_2}}|\nabla\phi\cdot\left(w_\varepsilon^{r_2}-u_\varepsilon\right)|^2
\leq C(r_2\varsigma(r_2,\varepsilon/r_2))^{-2}\int_{P_{r_2}}|w_\varepsilon^{r_2}-u_\varepsilon|^2
\leq C\int_{P_{r_2}}|\nabla(w_\varepsilon^{r_2}-u_\varepsilon)|^2.
\end{equation*}

\noindent
Let $r_2^*=r_2\varsigma(r_2,\varepsilon/r_2)\geq \varepsilon^*=\varepsilon\zeta(\varepsilon,1)$, then according to Fubini's theorem and \eqref{3.15}, we have
\begin{equation}\label{3.32}\begin{aligned}
\left(\frac{1}{|T^{\varepsilon,+}_{r_2}|}\fint_{I_{r_2}}\int_{P_{r_2}}|\nabla u_\varepsilon|^2\right)^{1/2}
\leq &\ C\left(\frac{1}{|T^{\varepsilon,+}_{r_2}|}\fint_{I_{r_2}}\int_{P_{r_2}}\fint_{Q_{r_2^*}(x,t)}|\nabla u_\varepsilon(y,s)|^2dydsdxdt\right)^{1/2}\\[7pt]
=& \ C\left(\frac{1}{|T^{\varepsilon,+}_{r_2}|}\fint_{I_{r_2}}\int_{P_{r_2}}\mathcal{M}_{r_2^*}^2[\nabla u_\varepsilon]^2\right)^{1/2}\\[7pt]
\leq& \ C\left(\frac{|P_{r_2}|}{|T^{\varepsilon,+}_{r_2}|}\right)^{1/2-1/p_1}
\left(\fint_{T^{\varepsilon,+}_{r_2}\times I_{r_2}}\mathcal{M}_{r_2^*}^2[\nabla u_\varepsilon]^{p_1}\right)^{1/p_1}\\[5pt]
\leq &\ C\varsigma(r_2,\varepsilon/r_2)^\gamma\left(\fint_{T^{\varepsilon,+}_{9r_2}\times I_{9r_2}}\mathcal{M}_{r_2^*}^2[\nabla u_\varepsilon]^{2}\right)^{1/2}\\[5pt]
\leq &\ C\varsigma(r_2,\varepsilon/r_2)^\gamma \left(\fint_{T^{\varepsilon,+}_{10r_2}\times I_{10r_2}}|\nabla u_\varepsilon|^{2}\right)^{1/2}.\\[5pt]
\end{aligned}\end{equation}

Now the desired inequality \eqref{3.27} follows readily from \eqref{3.31}-\eqref{3.32}.
\end{proof}

\section{Large-scale Lipschitz estimate}\label{l4}
\noindent
Let $u_\varepsilon \in L^2(-4,0;H^1(\Omega^\varepsilon_2))$ be a weak solution of
\begin{equation}\label{4.1}
\left\{
\begin{aligned}
\partial_tu_\varepsilon-\nabla\cdot (A^\varepsilon \nabla u_\varepsilon) &= 0 \qquad &\text{in } &Q^\varepsilon_{2}, \\
u_\varepsilon  &=0  \qquad &\text{on } & \Delta^\varepsilon_{2}.
\end{aligned}
\right.
\end{equation}
Now, we define two quantities $\Phi$ and $H$ as follows: for any $v\in L^2(-r^2,0;H^1(\Omega^\varepsilon_r))$,
\begin{equation}\label{4.2}
\Phi(r;v)  =:\frac{1}{r} \bigg( \fint_{Q^\varepsilon_r } |v|^2 \bigg)^{1/2}
\end{equation}
and
\begin{equation}\label{4.3}
H(r;v) = :\frac{1}{r}\inf_{k\in\mathbb{R}^d} \bigg( \fint_{Q^\varepsilon_{r} } |v - (n_r\cdot x)k|^2 \bigg)^{1/2}.
\end{equation}

\noindent For simplicity, denote $\Phi(r) = \Phi(r;u_\varepsilon), H(r) = H(r;u_\varepsilon)$. It follows from the Poincar\'{e} inequality and Caccioppoli inequalities \eqref{3.2} that the large-scale Lipschitz estimate \eqref{1.11} is equivalent to the estimate of $\Phi(r)$ for $r\in (\varepsilon,1)$.

To proceed,
we introduce some basic estimates of $\Phi$ and $H$. First of all, for any $r>0$, it is easy to check 
\begin{equation}\label{4.4}
H(r) \le \Phi(r).
\end{equation}
Moreover, using the $\varepsilon$-scale flatness of $Q^\varepsilon$ in Definition \ref{d1.1}, for any $r\in(\varepsilon,1)$, a direct computation yields that
\begin{equation}\label{4.5}
\sup_{r\le s\le 2r}\Phi(s) \le C\Phi(2r),
\end{equation}
for any $r\in (\varepsilon,1)$.

The outer normal $n_r$ of the flat boundary of $T^{\varepsilon,\pm}_r$ defined in Definition \ref{d1.1} will play an important role in obtaining the excess estimates. And the following lemma shows that $n_r$ changes gently associated with $r\in (\varepsilon,1)$, whose proof can be found in \cite[Lemma 3.1]{MR4199276}.

\begin{lemma}\label{l4.1}
Let $\varepsilon\leq r_1\leq r_2\leq 1$, then there exists a constant $C$, independent of $r_1$, $r_2$ and $\varepsilon$, such that
$$|n_{r_1}-n_{r_2}|\leq C\frac{r_2\varsigma(r_2,\varepsilon/r_2)}{r_1}.$$
\end{lemma}

\noindent
Moreover, for the quantities $\Phi$ and $H$, the following properties hold:

\begin{lemma}\label{l4.2}
	There exists a function $h:(0,2) \mapsto [0,\infty)$ such that for any $r\in (\varepsilon,1)$,
	\begin{equation}\label{4.6}
	\left\{
	\begin{aligned}
	h(r) & \le C(H(r) + \Phi(r)) \\[5pt]
	\Phi(r) &\le H(r) + h(r) \\[5pt]
	\sup_{r\leq r_1,r_2\leq  2r} |h(s_1) - h(s_2)| &\le CH(2r) + C\zeta(2r,\varepsilon/2r)\Phi(2r),
	\end{aligned}\right.
	\end{equation}
where the constant $C$ does not dependent on $r,r_1,r_2$ or $\varepsilon$.
\end{lemma}

\begin{proof} The proof is totally similar to \cite[Lemma 3.2]{MR4199276} for the elliptic case and we provide it for completeness.

	Let $k_r$ be the vector that minimizes $H(r)$, then
	\begin{equation}\label{4.7}
	H(r) = \frac{1}{r} \bigg( \fint_{Q^\varepsilon_{r} } |u_\varepsilon - (n_r\cdot x)k_r|^2 \bigg)^{1/2}.
	\end{equation}
	Now define $h(r) = |k_r|$, we will show that $h(r)$ is the desired function. To see the first inequality of \eqref{4.6}, using the $\varepsilon$-scale flatness of $Q^\varepsilon_r$ in Definition \ref{d1.1} for $r>\varepsilon$, we have
	\begin{equation*}
	\bigg( \fint_{Q_r^\varepsilon} |n_r\cdot x|^2 \bigg)^{1/2} \ge cr,
	\end{equation*}
	for some absolute constant $c>0$. Therefore, a direct computation shows that
	\begin{equation*}
	\begin{aligned}
	h(r) & \le |k_r| \frac{C}{r}\bigg( \fint_{Q_r^\varepsilon} |n_r\cdot x|^2 \bigg)^{1/2} = \frac{C}{r}\bigg( \fint_{Q_r^\varepsilon} |(n_r\cdot x)k_r|^2 \bigg)^{1/2} \\[5pt]
	& \le \frac{C}{r} \bigg( \fint_{Q^\varepsilon_{r} } |u_\varepsilon - (n_r\cdot x)k_r|^2 \bigg)^{1/2} + \frac{C}{r} \bigg( \fint_{Q^\varepsilon_{r} } |u_\varepsilon|^2 \bigg)^{1/2},
	\end{aligned}
	\end{equation*}
	which completes the proof of the first inequality of \eqref{4.6}. Moreover, the second inequality of \eqref{4.6} follows easily by the triangle inequality.
	
	Now, we need to prove the third inequality of \eqref{4.6}. Due to the first inequality of \eqref{4.6} and \eqref{4.4}, we have $h(r) \le C\Phi(r)$. Due to $r\in (\varepsilon,1)$, the $\varepsilon$-scale flatness condition in Definition \ref{d1.1} implies that $|Q^\varepsilon_r|\simeq r^{d+2}$ and $|n_{2r}\cdot x| \geq Cr>0$ in a subset of $\Omega^\varepsilon_r$ with volume comparable to $r^d$.
	Thus, for any $r_1,r_2\in [r,2r]$, a direct computation yields that
	\begin{equation}\label{4.8}
	\begin{aligned}
	|k_{s_1} - k_{s_2}| &\le |k_{s_1} - k_{s_2}|\cdot\frac{C}{r} \bigg( \fint_{Q^\varepsilon_{r} } |n_{2r}\cdot x|^2 \bigg)^{1/2}\le \frac{C}{r} \bigg( \fint_{Q^\varepsilon_{r} } |(n_{2r}\cdot x)\left(k_{s_1} - k_{s_2}\right)|^2 \bigg)^{1/2} \\
	& \le \frac{C}{r} \bigg( \fint_{Q^\varepsilon_{r} } |u_\varepsilon - (n_{2r}\cdot x)k_{s_1})|^2 \bigg)^{1/2} + \frac{C}{r} \bigg( \fint_{Q^\varepsilon_{r} } |u_\varepsilon - (n_{2r}\cdot x)k_{s_2})|^2 \bigg)^{1/2}.
	\end{aligned}
	\end{equation}
For the first term in \eqref{4.8}, according to Lemma \ref{l4.1}, $h(r)\leq C \Phi(r)$, \eqref{4.5} and the triangle inequality, we have
	\begin{equation*}
	\begin{aligned}
	& \frac{1}{r} \bigg( \fint_{Q^\varepsilon_{r} } |u_\varepsilon - (n_{2r}\cdot x)k_{s_1}|^2 \bigg)^{1/2} \\
	&\le \frac{C}{r} \bigg( \fint_{Q^\varepsilon_{s_1} } |u_\varepsilon - (n_{s_1}\cdot x)k_{s_1}|^2 \bigg)^{1/2} + |n_{2r}-n_{s_1}||k_{s_1}| \\[5pt]
	& \le \frac{C}{s_1} \inf_{k\in \mathbb{R}^d}\bigg( \fint_{Q^\varepsilon_{s_1} } |u_\varepsilon - (n_{s_1}\cdot x)k|^2 \bigg)^{1/2} + C\zeta(2r,\varepsilon/(2r)) \Phi(2r) \\[5pt]
	& \le \frac{C}{s_1}  \bigg( \fint_{Q^\varepsilon_{s_1} } |u_\varepsilon - (n_{s_1}\cdot x)k_{2r}|^2 \bigg)^{1/2} + C\zeta(2r,\varepsilon/(2r)) \Phi(2r) \\[5pt]
	& \le \frac{C}{2r} \bigg( \fint_{Q^\varepsilon_{2r} } |u_\varepsilon - (n_{2r}\cdot x)q_{2r}|^2 \bigg)^{1/2} + |n_{2r} - n_{s_1}| |k_{2r}| + C\zeta(2r,\varepsilon/2r) \Phi(2r) \\[5pt]
	& \le CH(2r) + C\zeta(2r,\varepsilon/(2r)) \Phi(2r).\\[5pt]
	\end{aligned}
	\end{equation*}
In fact, the estimate for the second term of \eqref{4.8} is similar as in the computation above. Therefore, for any $s_1,s_2\in [r,2r]$, there holds
	\begin{equation*}
	|h(s_1) - h(s_2)| \le |k_{s_1} -k_{s_2}| \le CH(2r) + C\zeta(2r,\varepsilon/(2r)) \Phi(2r),
	\end{equation*}
	which completes the proof of the third inequality of \eqref{4.6}.
\end{proof}

Next, we introduce the following approximation result. Note that, in Theorem \ref{t4.3}, the operator $\partial_t+\mathcal{L}_0$ denotes the homogenized operator for $\partial_t+\mathcal{L}_\varepsilon$ with $0<\varepsilon<1$ \cite{MR2839402}.
\begin{theorem}\label{t4.3}
Let $\varepsilon^*=\varepsilon\zeta(\varepsilon,1)$, $r_1=2\varepsilon^*\leq \varepsilon/2$ and $r_2\in (9\varepsilon^*,1/9)$. Suppose that $u_\varepsilon\in L^2(-4r_2^2,0; H^1(\Omega^\varepsilon_{2r_2}))$  satisfies the parabolic problem \eqref{4.1}, then there exists a function $u_0$ satisfying $(\partial_t+\mathcal{L}_0)u_0=0$ in $T^{\varepsilon,+}_{r_2}\times I_{r_2}$ with $u_0=0$ on $\partial T^{\varepsilon,+}_{r_2}\times [-r_2^2,0]$, such that
\begin{equation}\label{4.9}
\left(\fint_{T^{\varepsilon,+}_{r_2}\times I_{r_2}}|u_\varepsilon-u_0|^2\right)^{1/2}\leq Cr_2 \left(\frac{\varepsilon}{r_2}\right)^{\sigma}\left(\fint_{Q^\varepsilon_{10r_2}}|\nabla u_\varepsilon|^2\right)^{1/2},
\end{equation}
where $\sigma\in(0,\gamma]$ with $\gamma>0$ defined in Lemma \ref{l3.5} and the constant $C>0$ depends only on $\Lambda$, $d$ and $m$.\end{theorem}
\begin{proof}
Actually, the proof of Theorem \ref{t4.3} is similar to \cite[Lemma 6.2]{MR4072212}, therefore, we need only point out the main differences. By dilation, we may assume that $r_2=1$.

To proceed, let $u_0$ be the weak solution of the following initial-Dirichlet problem
\begin{equation} \label{4.10}
\left\{\begin{aligned}
\partial_t u_0 - \operatorname{div}\left(\widehat{A} \nabla u_0 \right)& = 0 \quad \ \text{ in } T^{\varepsilon,+}_{1}\times I_{1},\\
u_0&=u_\varepsilon \quad\text{ on } \partial _p \left(T^{\varepsilon,+}_{1}\times I_{1}\right),\\
\end{aligned}\right.
\end{equation}
where $\partial _p \left(T^{\varepsilon,+}_{1}\times I_{1}\right)$ denotes the parabolic boundary of the cylinder $T^{\varepsilon,+}_{1}\times I_{1}$.

First, by the standard energy estimates for parabolic operators after in view of the assumption \eqref{1.6}, there holds
\begin{equation}\label{*}
\begin{aligned}
\int_{T^{\varepsilon,+}_{1}\times I_{1}}|\nabla u_0|^2 \leq& C\int_{T^{\varepsilon,+}_{1}\times I_{1}}|\nabla u_\varepsilon|^2+C||\partial_t u_\varepsilon||_{L^2(-16,0;H^{-1}(T^{\varepsilon,+}_{1}))}^2\\[5pt]
\leq & C\int_{T^{\varepsilon,+}_{1}\times I_{1}}|\nabla u_\varepsilon|^2.
\end{aligned}\end{equation}

Moreover, according to Lemma \ref{l3.4} and \eqref{2.3}, we have
\begin{equation}\label{4.11}
			\bigg( \fint_{T^{\varepsilon,+}_{1}\times I_{1}} | \mathcal{M}_{r_1}^2 [\nabla u_0]|^{p_1}  \bigg)^{1/p_1}
			 \leq C\bigg( \fint_{T^{\varepsilon,+}_{1}\times I_{1}} | \mathcal{M}_{r_1}^2 [\nabla u_\varepsilon]|^2  \bigg)^{1/2}\leq C\bigg( \fint_{Q_2^\varepsilon} | \nabla u_\varepsilon|^2  \bigg)^{1/2}.
	\end{equation}

To proceed, note that a key estimate (different from our case) in \cite[Lemma 6.2]{MR4072212} is that Gen and Shen used the fact $\nabla u_0\in L^p(Q_1)$ for some $p>2$ (i.e., the Meyers estimate) to bound the errors and obtain the excess decay estimates with the help of the so-called $\varepsilon$-smoothing methods and the so-called dual correctors. While in our case, due to the irregularity of the boundary, we can only use
$\mathcal{M}_{r_1}^2 [\nabla u_0]\in L^{p_1}(T^{\varepsilon,+}_{1}\times I_{1})$ to bound the excess decay estimate. Now, compared with \cite[Theorem 3.1 and Lemma 6.2]{MR4072212}, a careful detection shows that we need only to estimate the following two terms:
\begin{equation}\label{4.12}
\int_{T^{\varepsilon,+}_{1}\times I_{1}\backslash
T^{\varepsilon,+}_{1-10\varepsilon}\times
I_{1-10\varepsilon}}|\nabla u_0|^2\quad \text{ and }\quad \varepsilon^2 \int_{T^{\varepsilon,+}_{1}\times
I_{1}\backslash T^{\varepsilon,+}_{1-9\varepsilon}\times I_{1-9\varepsilon}}\left(|\nabla^2
u_0|^2+|\partial_t u_0|^2\right),
\end{equation}
by virtue of $\mathcal{M}_{r_1}^2 [\nabla u_0]\in L^{p_1}(T^{\varepsilon,+}_{1}\times I_{1})$.

For the first term in \eqref{4.12}, similar as in \eqref{3.32}, it follows from Fubini's theorem, \eqref{3.15}, \eqref{4.11} and H\"older's inequality that
\begin{equation}\label{4.13}
\begin{aligned}
\int_{T^{\varepsilon,+}_{1}\times I_{1}\backslash
T^{\varepsilon,+}_{1-10\varepsilon}\times I_{1-10\varepsilon}}|\nabla u_0|^2
\leq& C\int_{T^{\varepsilon,+}_{1}\times I_{1}\backslash
T^{\varepsilon,+}_{1-10\varepsilon}\times I_{1-10\varepsilon}}|\mathcal{M}_{r_1}^2 [\nabla u_0]|^2\\[5pt]
\leq& C\varepsilon^{1-\frac2{p_1}}\left( \int_{T^{\varepsilon,+}_{1}\times I_{1}} | \mathcal{M}_{r_1}^2 [\nabla u_0]|^{p_1}  \right)^{2/p_1}\\[5pt]
\leq& C\varepsilon^{1-\frac2{p_1}}  \int_{T^{\varepsilon,+}_{10}\times I_{10}} | \nabla u_\varepsilon|^2  .
\end{aligned}\end{equation}

For the second term in \eqref{4.12}, using the idea in \eqref{4.13} and the standard regularity estimates for parabolic systems with constant coefficients \cite{MR241822}, there holds
\begin{equation*}\begin{aligned}
&\varepsilon^2 \int_{T^{\varepsilon,+}_{1}\times
I_{1}\backslash T^{\varepsilon,+}_{1-9\varepsilon}\times I_{1-9\varepsilon}}\left(|\nabla^2
u_0|^2+|\partial_t u_0|^2\right)\\[5pt]
\leq&\  C\varepsilon^2\int_{T^{\varepsilon,+}_{1}\times
I_{1}\backslash T^{\varepsilon,+}_{1-8\varepsilon}\times I_{1-8\varepsilon}}\frac{|\nabla
u_0(y,s)|^2dyds}{|\text{dist}_p((y,s),\partial_p (T^{\varepsilon,+}_{1}\times
I_{1}))|^2}\\[5pt]
\leq&\  C\int_{T^{\varepsilon,+}_{1}\times
I_{1}\backslash T^{\varepsilon,+}_{1-8\varepsilon}\times I_{1-8\varepsilon}}|\nabla
u_0(y,s)|^2dyds\\[5pt]
\leq&\  C\varepsilon^{1-\frac2{p_1}}\left( \int_{T^{\varepsilon,+}_{1}\times I_{1}} | \mathcal{M}_{r_1}^2 [\nabla u_0]|^{p_1}  \right)^{2/p_1}\\[5pt]
\leq&\  C\varepsilon^{1-\frac2{p_1}}  \int_{T^{\varepsilon,+}_{10}\times I_{10}} | \nabla u_\varepsilon|^2  .
\end{aligned}\end{equation*}
Consequently, we complete the estimates of the terms in \eqref{4.12} and the desired estimates \eqref{4.9} can be obtained. For a more detailed proof, refer to \cite[Theorem 3.1 and Lemma 6.2]{MR4072212}.
\end{proof}

Similar as in \cite{MR4199276}, for $0<a\leq 1$, denote
 $$\tilde{H}(r,a,w_\varepsilon^{r})=\inf_{k\in \mathbb{R}^d}\frac1{ar}\left(\fint_{I_{ar}}\fint_{T^{\varepsilon,+}_r\cap B_{ar}}|w_\varepsilon^{r}-(n_r\cdot x)k|^2\right)^{1/2}.$$
where $w_\varepsilon^{r}$ satisfies
\begin{equation*}
\left\{\begin{aligned}
\partial_t w_\varepsilon^r - \operatorname{div}\left(A^\varepsilon \nabla w_\varepsilon^r \right)&= 0 \quad\quad\text{ in }T^{\varepsilon,+}_{r}\times I_{r},\\
w_\varepsilon^r&=u_\varepsilon \ \ \quad\text{ on }\partial _p \left(T^{\varepsilon,+}_{r}\times I_{r}\right).\\
\end{aligned}\right.
\end{equation*}

Note that we will drop the superscript $r$ for simplicity, if there is no ambiguity.
Now, using the smoothness of $u_0$ near the flat boundary $\partial T^{\varepsilon,+}_{r}\times I_{r}$, we can show that
\begin{lemma}\label{l4.4}
For any $r\in (3\varepsilon,1/10)$, there exists a constant $\theta\in (0,1)$ such that
$$
\tilde{H}(r,\theta,w_\varepsilon)\leq \frac 1 2 \tilde{H}(r,1,w_\varepsilon)+C\left((\varepsilon/r)^\sigma+\varsigma (r,\varepsilon/r)\right)\Phi(10r),
$$
with $\sigma>0$ given in Theorem \ref{t4.3}.\end{lemma}
\begin{proof}
The elliptic version of divergence form with rough boundaries has been proved by Zhuge \cite{MR4199276}.

 Let $u_0$ be given by Theorem \ref{t4.3}, then
by the $C^{1,1}$ regularity of $u_0$ on the flat boundary, we know that
\begin{equation}\label{4.14}
||(\nabla^2 u_0,\partial_t u_0, r^{-1}\nabla u_0)||_{L^\infty((T^{\varepsilon,+}_r\cap B_{r/4})\times I_{r/4})}\leq Cr^{-1}\left(\fint_{I_{r/2}}\fint_{T^{\varepsilon,+}_{r}\cap B_{r/2}}|\nabla u_0|^2\right),\end{equation}
where we use the notation $||(A,B,C)||_{\mathcal{B}}=:||A||_{\mathcal{B}}+||A||_{\mathcal{B}}+||A||_{\mathcal{B}}$ in the inequality above.

To proceed, let $x_r$ be the point on the flat boundary $\partial T^{\varepsilon,+}_r\cap B_r$ such that it is the closet point to the origin. By assumption \eqref{1.3}, it is easy to see that $|x_r|\leq Cr\zeta(r,\varepsilon/r)$. Now, since $u_0(\cdot,0)$ is identically 0 on the flat boundary $\partial T^{\varepsilon,+}_r\cap B_{r}$, the tangential derivatives vanish at $x_r$, i.e.,
$$(I-n_r\otimes n_r)\nabla u_0(x_r,0)=0,$$
which implies that
\begin{equation*}
\nabla u_0(x_r,0)=(n_r\otimes n_r)\nabla u_0(x_r,0)=n_r(n_r\cdot \nabla u_0(x_r,0)).
\end{equation*}

\noindent
Hence, it follows from the Taylor expansion of $u_0$ at $(x_r,0)$ that
$$\begin{aligned}&|u_0(x,t)-u_0(x_r,0)-(x-x_r)\cdot \nabla u_0(x_r,0)|\\[6pt]
&\quad\leq C|x-x_r|^2||\nabla^2 u_0||_{L^\infty((T^{\varepsilon,+}_r\cap B_{r/4})\times I_{r/4})}\\[6pt]
&\quad\quad\quad+C|t|\cdot||\partial_t u_0||_{L^\infty((T^{\varepsilon,+}_r\cap B_{r/4})\times I_{r/4})}.\\[6pt]
\end{aligned}$$

\noindent
Then for any $(x,t)\in \left(T^{\varepsilon,+}_r\cap B_{r/4}\right)\times I_{r/4}$, there holds
$$\begin{aligned}
&|u_0(x,t)-(x\cdot n_r)(n_r\cdot \nabla u_0(x_r,0))|\\[8pt]
&\quad \leq C|x-x_r|^2||\nabla^2 u_0||_{L^\infty((T^{\varepsilon,+}_r\cap B_{r/4})\times I_{r/4})}\\
&\quad\quad+C|t|||\partial_t u_0||_{L^\infty((T^{\varepsilon,+}_r\cap B_{r/4})\times I_{r/4})}
+C|x_r|\left(\fint_{I_r}\fint_{T^{\varepsilon,+}_{r/2}\cap B_{r/2}}|\nabla u_0|^2\right)^{1/2}.
\end{aligned}$$

\noindent
Thus, for any $\theta\in (0,1/4)$, we have

\begin{equation}\label{4.15}\begin{aligned}
&\inf_{k\in\mathbb{R}^d}\frac 1 {\theta r}\left(\fint_{I_{\theta r}}\fint_{T^{\varepsilon,+}_{r}\cap B_{\theta r}}|u_0-(n_r\cdot x)k|^2\right)^{1/2}\\[8pt]
%\leq &(C\theta r+ C_\theta |x_r|)||(\nabla^2 u_0,\partial_t u_0)||_{L^\infty((T^{\varepsilon,+}_r\cap B_{r/4})\times I_{r/4})}+C_\theta \zeta(r,\varepsilon/r)\left(\fint_{I_r}\fint_{T^{\varepsilon,+}_{r/2}\cap B_{r/2}}|\nabla u_0|^2\right)^{1/2}\\
\leq &C\theta r||(\nabla^2 u_0,\partial_t u_0)||_{L^\infty((T^{\varepsilon,+}_r\cap B_{r/4})\times I_{r/4})}+C_\theta \zeta(r,\varepsilon/r)\left(\fint_{I_r}\fint_{T^{\varepsilon,+}_{r/2}\cap B_{r/2}}|\nabla u_0|^2\right)^{1/2}.
\end{aligned}\end{equation}

\noindent
Moreover, for any $k\in\mathbb{R}^d$, it is easy to see that $u_0-(n_r\cdot(x-x_r))k$ is also a weak solution to the following parabolic problem:
\begin{equation} \label{4.16}
\left\{\begin{aligned}
(\partial_t + \mathcal{L}_0)\left(u_0-(n_r\cdot(x-x_r))k\right)& = 0 \quad\quad\quad \text{ in } T^{\varepsilon,+}_{r}\times I_{r},\\
u_0-(n_r\cdot(x-x_r))k&=0 \quad\quad\quad\text{ on } \left(\partial T^{\varepsilon,+}_{r}\cap B_r\right)\times I_{r}.\\
\end{aligned}\right.
\end{equation}

\noindent
To proceed, applying the $C^{1,1}$ estimate \eqref{4.14} for $u_0-(n_r\cdot(x-x_r))k$ and the parabolic Caccioppli inequality \eqref{3.2} yields that
\begin{equation}\label{4.17}
\begin{aligned}
r ||(\nabla^2 u_0,\partial_t u_0)||_{L^\infty((T^{\varepsilon,+}_r\cap B_{r/4})\times I_{r/4})}\leq C\inf_{k\in\mathbb{R}^d}\frac 1 r \left(\fint_{I_{3 r/4}}\fint_{T^{\varepsilon,+}_{r}\cap B_{3 r/4}}|u_0-(n_r\cdot(x-x_r))k|^2\right)^{1/2}.
\end{aligned}\end{equation}

\noindent Now, let $k_1$ satisfy

\begin{equation}\label{4.18}
\inf_{k\in\mathbb{R}^d}\frac 1 r \left(\fint_{I_{3 r/4}}\fint_{T^{\varepsilon,+}_{r}\cap B_{3r/ 4}}|u_0-(n_r\cdot x)k|^2\right)^{1/2}=\left(\fint_{I_{3 r/4}}\fint_{T^{\varepsilon,+}_{r}\cap B_{3 r/4}}|u_0-(n_r\cdot x)k_1|^2\right)^{1/2}.
\end{equation}
A direct geometrical observation implies that $|n_r\cdot x|\geq Cr$ in a large portion of $T^{\varepsilon,+}_r\cap B_{3r/4}$.  Thus, it directly follows from the triangle inequality that

\begin{equation}\label{4.19}
\frac1r\left(\fint_{I_{3 r/4}}\fint_{T^{\varepsilon,+}_{r}\cap B_{3 r/4}}|u_0-(n_r\cdot x)k_1|^2\right)^{1/2}\geq c|k_1|-\frac1 r\left(\fint_{I_{3 r/4}}\fint_{T^{\varepsilon,+}_{r}\cap B_{3 r/4}}|u_0|^2\right)^{1/2}.\end{equation}

\noindent Moreover, it is easy to see that
\begin{equation}\label{4.20}
\inf_{k\in\mathbb{R}^d}\frac1r\left(\fint_{I_{3 r/4}}\fint_{T^{\varepsilon,+}_{r}\cap B_{3 r/4}}|u_0-(n_r\cdot x)k_1|^2\right)^{1/2}\leq \frac1r\left(\fint_{I_{3 r/4}}\fint_{T^{\varepsilon,+}_{r}\cap B_{3 r/4}}|u_0|^2\right)^{1/2}.
\end{equation}

\noindent Now, combining \eqref{4.19}-\eqref{4.20} yields that
\begin{equation}\label{4.21}
|k_1|\leq \frac Cr\left(\fint_{I_{3 r/4}}\fint_{T^{\varepsilon,+}_{r}\cap B_{3 r/4}}|u_0|^2\right)^{1/2}.
\end{equation}

\noindent Consequently, due to \eqref{4.21}, a direct computation shows that
\begin{equation}\label{4.22}
\begin{aligned}
&\inf_{k\in\mathbb{R}^d}\frac 1 r \left(\fint_{I_{3 r/4}}\fint_{T^{\varepsilon,+}_{r}\cap B_{3 r/4}}|u_0-(n_r\cdot(x-x_r))k|^2\right)^{1/2}\\
&\hspace{6.5cm}-\frac 1 r \left(\fint_{I_{3 r/4}}\fint_{T^{\varepsilon,+}_{r}\cap B_{3 r/4}}|u_0-(n_r\cdot x)k_1|^2\right)^{1/2}\\
&\leq \inf_{k\in\mathbb{R}^d}\frac 1 r \left(\fint_{I_{3 r/4}}\fint_{T^{\varepsilon,+}_{r}\cap B_{3 r/4}}|(n_r\cdot x)(k-k_1)-(n_r\cdot x_r)k|^2\right)^{1/2}\\
&\leq \frac 1 r \left(\fint_{I_{3 r/4}}\fint_{T^{\varepsilon,+}_{r}\cap B_{3 r/4}}|(n_r\cdot x_r)k_1|^2\right)^{1/2}\\
&\leq C\frac{|x_r|}{r^2}\left(\fint_{I_{3 r/4}}\fint_{T^{\varepsilon,+}_{r}\cap B_{3 r/4}}|u_0|^2\right)^{1/2},
\end{aligned}\end{equation}
which implies that
\begin{equation}\label{4.23}
\begin{aligned}
&\inf_{k\in\mathbb{R}^d}\frac 1 r \left(\fint_{I_{3 r/4}}\fint_{T^{\varepsilon,+}_{r}\cap B_{3 r/4}}|u_0-(n_r\cdot(x-x_r))k|^2\right)^{1/2}\\
\leq &\inf_{k\in\mathbb{R}^d}\frac 1 r \left(\fint_{I_{3 r/4}}\fint_{T^{\varepsilon,+}_{r}\cap B_{3 r/4}}|u_0-(n_r\cdot x)k|^2\right)^{1/2}+C\frac{|x_r|}{r^2}\left(\fint_{I_{3 r/4}}\fint_{T^{\varepsilon,+}_{r}\cap B_{3 r/4}}|u_0|^2\right)^{1/2}.
\end{aligned}\end{equation}
To proceed, it follows from \eqref{4.17} and \eqref{4.23} that
\begin{equation}\label{4.24}
\begin{aligned}
r ||(\nabla^2 u_0,\partial_t u_0)||_{L^\infty((T^{\varepsilon,+}_r\cap B_{r/4})\times I_{r/4})}\leq & C\inf_{k\in\mathbb{R}^d}\frac 1 r \left(\fint_{I_{3 r/4}}\fint_{T^{\varepsilon,+}_{r}\cap B_{3 r/4}}|u_0-(n_r\cdot x)k|^2\right)^{1/2}\\
&+C\frac{|x_r|}{r^2}\left(\fint_{I_{3 r/4}}\fint_{T^{\varepsilon,+}_{r}\cap B_{3 r/4}}|u_0|^2\right)^{1/2}.
\end{aligned}\end{equation}

\noindent
Now, it follows from \eqref{4.15}, \eqref{4.24} and Poincar\'e inequality that
$$\begin{aligned}
&\inf_{k\in\mathbb{R}^d}\frac 1 {\theta r}\left(\fint_{I_{\theta r}}\fint_{T^{\varepsilon,+}_{r}\cap B_{\theta r}}|u_0-(n_r\cdot x)k|^2\right)^{1/2}\\[5pt]
\leq& C\theta \inf_{k\in\mathbb{R}^d}\frac 1 r \left(\fint_{I_{3 r/4}}\fint_{T^{\varepsilon,+}_{r}\cap B_{3 r/4}}|u_0-(n_r\cdot x)k|^2\right)^{1/2}+C_\theta \zeta(r,\varepsilon/r)\left(\fint_{I_{3r/4}}\fint_{T^{\varepsilon,+}_{r}\cap B_{3r/4}}|\nabla u_0|^2\right).
\end{aligned}$$

\noindent According to \eqref{4.9} and basic energy estimates \eqref{*} for $u_0$, we have
$$\begin{aligned}
\inf_{k\in\mathbb{R}^d}\frac 1 {\theta r}\left(\fint_{I_{\theta r}}\fint_{T^{\varepsilon,+}_{r}\cap B_{\theta r}}|u_\varepsilon-(n_r\cdot x)k|^2\right)&^{1/2}
\leq C\theta \inf_{k\in\mathbb{R}^d}\frac 1 { r}\left(\fint_{I_{ r}}\fint_{T^{\varepsilon,+}_{r}}|u_\varepsilon-(n_r\cdot x)k|^2\right)^{1/2}\\
&+C\left(\left({\varepsilon}/{r}\right)^\sigma+\zeta(r,\varepsilon/r)\right)
\left(\fint_{I_{5 r}}\fint_{\Omega_{5r}^\varepsilon}|\nabla u_\varepsilon|^2\right)^{1/2}.
\end{aligned}$$
%where we have used the energy estimates in the inequality above.

 Consequently, due to the Caccioppli inequality \eqref{3.2} after choosing $\theta$ so small such that $C\theta=\frac 12$, we obtain the desired estimate.
\end{proof}

Now, similar as in \cite[Lemma 3.4]{MR4199276}, we have the following approximation result.
\begin{lemma}\label{l4.5}
 There are $\theta\in (0,1)$ and $\varepsilon_0\in (0,1)$ such that if $r\in (\varepsilon/\varepsilon_0,\varepsilon_0)$
	\begin{equation*}
	H(\theta r) \le \frac{1}{2}H(r) + C\left((\varepsilon/r)^\sigma+\varsigma (r,\varepsilon/r)^{1/2}\right) \Phi(40r),
	\end{equation*}
where $\sigma>0$ is given in Theorem \ref{t4.3}.
\end{lemma}
\begin{proof} The proof is almost identical to \cite[Lemma 3.4]{MR4199276} after obtaining Theorem \ref{t4.3} and Lemma \ref{l4.4}, and we provide it in the Appendix for completeness.
\end{proof}

To complete the proof of Theorem \ref{t1.4}, we need the following
general iteration lemma, whose proof can be found in \cite[Lemma 3.5]{MR4199276}.
\begin{lemma}\label{l4.6}
	Suppose $\eta: (0,1]\times (0,1]\mapsto [0,1]$ is an admissible modulus.
	Let $H,\Phi,h:(0,2]\mapsto [0,\infty)$ be nonnegative functions. Suppose that there exist $\theta\in (0,1/4), \varepsilon_0\in (0,\theta)$ and $C_0>0$ so that $H,\Phi$ and $h$ satisfy:
	\begin{itemize}
		\item For every $r\in (\varepsilon/\varepsilon_0,\varepsilon_0)$,
		\begin{equation}\label{4.30}
		H(\theta r) \le \frac{1}{2} H(r) + C_0 \big\{ \eta(r,\varepsilon/r)\big\} \Phi(40r).
		\end{equation}
		
		\item For every $r\in (\varepsilon,1)$,
\begin{equation}\label{4.31}
			\begin{aligned}
			H(r) & \le C_0 \Phi(r), \\[5pt]	
			h(r) & \le C_0 \big( H(r) + \Phi(r) \big), \\[5pt]
			\Phi(r) &\le C_0\big( H(r) + h(r)\big), \\[5pt]
			\sup_{r\le t\le 2r} \Phi(t) & \le C_0\Phi(2r), \\[5pt]
			\sup_{r\le s,t\le 2r} |h(s) - h(t)| &\le C_0 H(2r). \\[5pt]
			\end{aligned}
		\end{equation}
		%\begin{subequations}\label{est.HTP}
%			\begin{align}
%			H(r) & \le C_0 \Phi(r) \label{4.29:a}\\	
%			h(r) & \le C_0 \big( H(r) + \Phi(r) \big) \label{4.29:b}\\
%			\Phi(r) &\le C_0\big( H(r) + h(r)\big) \label{4.29:c}\\
%			\sup_{r\le t\le 2r} \Phi(t) & \le C_0\Phi(2r) \label{4.29:d}\\
%			\sup_{r\le s,t\le 2r} |h(s) - h(t)| &\le C_0 H(2r) \label{4.29:e}
%			\end{align}
%		\end{subequations}
	\end{itemize}
	Then
	\begin{equation}\label{4.32}
	\int_{\varepsilon}^{1} \frac{H(r)}{r} dr + \sup_{\varepsilon\le r\le 1} \Phi(r) \le C\Phi(2),
	\end{equation}
	where $C$ depends on the parameters except $\varepsilon$.
\end{lemma}

\noindent Now we are ready to give the proof of Theorem \ref{t1.4}.\\
\noindent \textbf{Proof of Theorem \ref{t1.4}} . Let $\Phi$ and $H$ be defined in \eqref{4.2} and \eqref{4.3}, respectively, and let $h$ be given in Lemma \ref{l4.2}. Now, define

$$H^*(r)=H(r)+\varsigma(r,\varepsilon/r)\Phi(r).$$

\noindent
Then, according to \eqref{4.4}-\eqref{4.6} and Lemma \ref{l4.5}, we know that $\Phi$, $H^*$ and $h$ satisfy the assumptions of Lemma \ref{l4.6} (with $H$ replaced by $H^*$) with $\eta(r,s)=s^\sigma+\varsigma(r,s)^{1/2}+\varsigma(\theta r,s/\theta)$ for $r\in(\varepsilon,1)$. To proceed, since $\varsigma$ is a $\sigma$-admissible modulus, then $\eta(r,s)$ is an admissible modulus. Now, it follows from Lemma \ref{4.6} that
$$\sup_{\varepsilon\leq r\leq 1}\Phi(r)\leq C\Phi(2).$$
Finally, the Poincar\'e inequality and the Caccioppli inequality \eqref{3.2} yield the desired estimate \eqref{1.11}.

\appendix
\section{Appendix of proof of Lemma \ref{l4.5}}
In this appendix, we provide the details of the proof of Lemma \ref{l4.5}.\\

 It directly follows from the triangle inequality that
	\begin{equation*}
	|H(r;f) - H(r;g)| \le H(r;f-g),
	\end{equation*}
	for any $f,g\in L^2(Q^\varepsilon_r;\mathbb{R}^d)$. Now, for $u_\varepsilon$ and $w_\varepsilon$, it follows from Lemma \ref{l4.4} that
	\begin{align}\label{4.25}
\notag	H(\theta r; u_\varepsilon) & \le H(\theta r; w_\varepsilon) + H(\theta r; u_\varepsilon - w_\varepsilon) \\[5pt]\notag
	& \le \widetilde{H}(r,\theta;w_\varepsilon) + ( H(\theta r;w_\varepsilon) - \widetilde{H}(r,\theta;w_\varepsilon) ) + H(\theta r; u_\varepsilon - w_\varepsilon)\\[5pt]\notag
	& \le \frac{1}{2} \widetilde{H}(r,1;w_\varepsilon) + ( H(\theta r;w_\varepsilon) - \widetilde{H}(r,\theta;w_\varepsilon) ) + H(\theta r; u_\varepsilon - w_\varepsilon) \\[5pt]
	& \qquad + C\left((\varepsilon/r)^\sigma+\varsigma (r,\varepsilon/r)\right) \Phi(10r) \\[5pt]\notag
	& \le \frac{1}{2} H(r; u_\varepsilon) + ( H(\theta r;w_\varepsilon) - \widetilde{H}(r,\theta;w_\varepsilon) ) + \frac{1}{2}(\widetilde{H}(r,1;w_\varepsilon) - H(r;w_\varepsilon))  \\\notag
	&\qquad + H(\theta r; u_\varepsilon - w_\varepsilon) +\frac12 H(r;u_\varepsilon - w_\varepsilon)
	+ C\left((\varepsilon/r)^\sigma+\varsigma (r,\varepsilon/r)\right) \Phi(10r).
	\end{align}

\noindent	
Recall that in the proof of \eqref{4.18}-\eqref{4.21}, for $0<\theta\leq 1$, the best constant $k_{\theta r}$ in $H(\theta r;w_\varepsilon)$ satisfies
 \begin{equation}\label{4.26}
 |k_{\theta r}|\leq \frac C {\theta r}\bigg( \fint_{Q^\varepsilon_{\theta r}} |w_\varepsilon|^2 \bigg)^{1/2},\end{equation} then there holds
	\begin{equation*}
	\begin{aligned}
	I_1 &:= H(\theta r;w_\varepsilon) - \widetilde{H}(r,\theta;w_\varepsilon) \\
	& = \frac{1}{\theta r} \inf_{k\in \mathbb{R}^d} \bigg( \fint_{Q^\varepsilon_{\theta r} } |w_\varepsilon - (n_{\theta r}\cdot x)k|^2 \bigg)^{1/2} - \frac{1}{\theta r}\inf_{k\in \mathbb{R}^d}  \bigg(\fint_{I_{\theta r}} \fint_{T_r^{\varepsilon,+} \cap B_{\theta r}} |w_\varepsilon - (n_r\cdot x) k|^2\bigg)^{1/2} \\
	& \le \frac{1}{\theta r} \inf_{k\in \mathbb{R}^d} \bigg( \fint_{Q^\varepsilon_{\theta r} } |w_\varepsilon - (n_{r}\cdot x)k|^2 \bigg)^{1/2} - \frac{1}{\theta r}\inf_{k\in \mathbb{R}^d}  \bigg( \fint_{I_{\theta r}}\fint_{T_r^{\varepsilon,+} \cap B_{\theta r}} |w_\varepsilon - (n_r\cdot x) k|^2\bigg)^{1/2} \\
	& \qquad +  \frac{C|n_r - n_{\theta r}|}{\theta r} \bigg( \fint_{Q^\varepsilon_{\theta r}} |w_\varepsilon|^2 \bigg)^{1/2}.
	\end{aligned}
	\end{equation*}
According to Lemma \ref{l3.1}, the Poincar\'{e} inequality, the Caccioppoli inequalities \eqref{3.21} and the energy estimates of $w_\varepsilon$, the last term in the above inequality is bounded by $C\zeta(r,\varepsilon/r) \Phi(20r) $. To proceed, a direct computation yields that
	\begin{equation*}
	\inf_{k\in \mathbb{R}^d} \bigg( \fint_{Q^\varepsilon_{\theta r} } |w_\varepsilon - (n_{r}\cdot x)k|^2 \bigg)^{1/2} \le \frac{|T_r^{\varepsilon,+}\cap B_{\theta r}|^{1/2}}{|\Omega^\varepsilon_{\theta r}|^{1/2}} \inf_{k\in \mathbb{R}^d} \bigg(\fint_{I_{\theta r}} \fint_{T_{r}^{\varepsilon,+} \cap B_{\theta r} } |w_\varepsilon - (n_{r}\cdot x)k|^2 \bigg)^{1/2}.
	\end{equation*}
	According the definition of $T_r^{\varepsilon,+}$ in \eqref{1.3} and the assumption that $r\varsigma(r,\varepsilon/r)$ is non-decreasing, there holds
	\begin{equation*}
	\frac{|T_r^{\varepsilon,+}\cap B_{\theta r}|^{1/2}}{|\Omega^\varepsilon_{\theta r}|^{1/2}} = \bigg(1 + \frac{|T_r^{\varepsilon,+}\cap B_{\theta r}\setminus \Omega^\varepsilon_{\theta r}|}{|\Omega^\varepsilon_{\theta r}|} \bigg)^{1/2} \le 1 + C\zeta(r,\varepsilon/r)^{1/2}.
	\end{equation*}
Consequently, combining the above three inequalities yields that
	\begin{equation}\label{4.27}
	I_1 \le C\zeta(r,\varepsilon/r)^{1/2} \Phi(20r).\\[8pt]
	\end{equation}

\noindent	
To proceed, we are ready to estimate the following term
	\begin{equation*}
	\begin{aligned}
	I_2 &:= \widetilde{H}(r,1;w_\varepsilon) - H(r;w_\varepsilon) \\
	& = \frac{1}{r}\inf_{k\in \mathbb{R}^d}  \bigg( \fint_{T_r^{\varepsilon,+}} |w_\varepsilon - (n_r\cdot x) k|^2\bigg)^{1/2} - \frac{1}{r} \inf_{k\in \mathbb{R}^d} \bigg( \fint_{Q^\varepsilon_{r} } |w_\varepsilon - (n_{ r}\cdot x)k|^2 \bigg)^{1/2}.
	\end{aligned}
	\end{equation*}
Similar as in \eqref{4.18}, we let $k_r$ be the vector satisfying
	\begin{equation*}
	H(r;w_\varepsilon) = \frac{1}{r} \bigg( \fint_{Q^\varepsilon_{r} } |w_\varepsilon - (n_{ r}\cdot x)k_r|^2 \bigg)^{1/2}.
	\end{equation*}

\noindent
In view of \eqref{4.26}, we have $|k_r| \le C\Phi(r;w_\varepsilon)$.
To proceed, by using $T_r^{\varepsilon,+}=\Omega^\varepsilon_r\cup \left(T_r^{\varepsilon,+}\setminus \Omega^\varepsilon_r\right)$, a direct computation shows that
	\begin{equation*}
	\begin{aligned}
	&\frac{1}{r}\inf_{k\in \mathbb{R}^d}  \bigg( \fint_{I_r}\fint_{T_r^{\varepsilon,+}} |w_\varepsilon - (n_r\cdot x) k|^2\bigg)^{1/2}
\leq \frac{1}{r} \bigg( \fint_{I_r}\fint_{T_r^{\varepsilon,+}} |w_\varepsilon - (n_r\cdot x) k_r|^2\bigg)^{1/2} \\[5pt]
	\leq& \frac{1}{r} \bigg( \fint_{Q^\varepsilon_r} |w_\varepsilon - (n_r\cdot x) k_r|^2\bigg)^{1/2} \frac{|\Omega^\varepsilon_r|^{1/2}}{ |T_r^{\varepsilon,+}|^{1/2} }+
	 \frac{1}{r} \bigg( \frac{1}{|T_r^{\varepsilon,+}|} \fint_{I_r}\int_{T_r^{\varepsilon,+}\setminus \Omega^\varepsilon_r} |w_\varepsilon - (n_r\cdot x) k_r|^2  \bigg)^{1/2}.
	\end{aligned}
	\end{equation*}
Recall that $\Omega_r^\varepsilon \subset T_r^{\varepsilon,+}$ and $T_r^{\varepsilon,+}\setminus \Omega^\varepsilon_r \subset T_r^{\varepsilon,+}\setminus T_r^{\varepsilon,-}$ and the boundary condition satisfied by $w_\varepsilon$, then it follows the Poincar\'{e} inequality that
	\begin{equation*}
	\bigg( \fint_{I_r}\int_{T_r^{\varepsilon,+}\setminus \Omega^\varepsilon_r} |w_\varepsilon|^2  \bigg)^{1/2} \le Cr\zeta(r,\varepsilon/r) \bigg( \fint_{I_r}\int_{T_r^{\varepsilon,+}\setminus T_r^{\varepsilon,-}} |\nabla w_\varepsilon|^2  \bigg)^{1/2}.
	\end{equation*}
	Combining the above two inequalities, the estimate \eqref{4.26} of $|k_r|$ and the energy estimates \eqref{3.22} of $w_\varepsilon$, we obtain
	\begin{equation*}
	\frac{1}{r}\inf_{k\in \mathbb{R}^d}  \bigg( \fint_{I_r}\fint_{T_r^{\varepsilon,+}} |w_\varepsilon - (n_r\cdot x) k|^2\bigg)^{1/2} \le \frac{1}{r} \bigg( \fint_{Q^\varepsilon_r} |w_\varepsilon - (n_r\cdot x) k_r|^2\bigg)^{1/2} + C\zeta(r,\varepsilon/r) \Phi(20r),
	\end{equation*}
	which directly implies that
	\begin{equation}\label{4.28}
	I_2 \leq  C\zeta(r,\varepsilon/r) \Phi(20r).\\[8pt]
	\end{equation}
	
\noindent	Finally, we need only to estimate $H(r;u_\varepsilon -w_\varepsilon)$, since the estimate of $H(\theta r; u_\varepsilon -w_\varepsilon)$ is similar as the estimate $H(r;u_\varepsilon -w_\varepsilon)$. Actually, for any $r\in (\varepsilon/\varepsilon_0,\varepsilon_0)$, it follows from Lemma \ref{l3.5} and the Poincar\'{e} inequality that
	\begin{equation}\label{4.29}
	\begin{aligned}
	H(r;u_\varepsilon -w_\varepsilon) & = \frac{1}{r} \inf_{q\in \mathbb{R}^d} \bigg( \fint_{Q^\varepsilon_{r} } |u_\varepsilon - w_\varepsilon - (n_{ r}\cdot x)q|^2 \bigg)^{1/2} \\
	& \le \frac{1}{r} \bigg( \fint_{Q^\varepsilon_{r} } |u_\varepsilon - w_\varepsilon|^2 \bigg)^{1/2} \\
	& \le C\left((\varepsilon/r)^\gamma+\varsigma (r,\varepsilon/r)\right) \bigg( \fint_{Q^\varepsilon_{20r} } |\nabla u_\varepsilon|^2 \bigg)^{1/2} \\
	& \le C\left((\varepsilon/r)^\gamma+\varsigma (r,\varepsilon/r)\right) \Phi(40r),
	\end{aligned}
	\end{equation}
with $\gamma>0$ given in Lemma \ref{l3.5}. Consequently, combining \eqref{4.27}-\eqref{4.29} yields the desired estimate in Lemma \ref{l4.5}.

\begin{center}\textbf{\large{Acknowlwdgement}}\end{center}
The work of P. Yu was supported by the Outstanding
Innovative Talents Cultivation Funded Programs 2023 of Renmin University of China and the work of Y. Zhang was supported by the
Hubei Provincial Natural Science Foundation of China under Grant 2024AFB357 and by the National Natural Science Foundation of China under Grant 12401256.

\normalem\bibliographystyle{plain}{}

%\normalem\bibliographystyle{plain}{}
%\bibliography{bib}
\end{document}